\documentclass[11pt, a4paper]{amsart}
\usepackage{amsmath, amsthm, amssymb}
\usepackage{txfonts, mathrsfs}
\usepackage[usenames]{color}
\usepackage{psfrag}
\usepackage{graphicx}
\usepackage{enumitem}
\usepackage{comment}
\usepackage{todonotes}
\numberwithin{equation}{section}

\newtheorem{thm}{Theorem}[section]
\newcommand{\bt}{\begin{thm}}
\newcommand{\et}{\end{thm}}

\newtheorem{cor}[thm]{Corollary}
\newcommand{\bc}{\begin{cor}}
\newcommand{\ec}{\end{cor}}

\newtheorem{lem}[thm]{Lemma}
\newcommand{\bl}{\begin{lem}}
\newcommand{\el}{\end{lem}}

\newtheorem{prop}[thm]{Proposition}
\newcommand{\bp}{\begin{prop}}
\newcommand{\ep}{\end{prop}}

\newtheorem{defn}[thm]{Definition}
\newcommand{\bd}{\begin{defn}}
\newcommand{\ed}{\end{defn}}

\newtheorem{rmrk}[thm]{Remark}
\newcommand{\br}{\begin{rmrk}}
\newcommand{\er}{\end{rmrk}}

\newtheorem{quest}[thm]{Question}
\newcommand{\bq}{\begin{quest}}
\newcommand{\eq}{\end{quest}}

\newtheorem{ex}[thm]{Example}
\newcommand{\bex}{\begin{ex}}
\newcommand{\eex}{\end{ex}}

\newcommand{\C}{\mathbb{C}}
\newcommand{\N}{\mathbb{N}}

\newcommand{\Cat}{\mathrm{CAT}}

\DeclareMathOperator{\arsinh}{arsinh}

\newdimen\vintkern\vintkern12pt
\def\vint{-\kern-\vintkern\int}

\newcommand{\hm}{{\mathcal H}}

\newcommand{\trace}{\operatorname{tr}}

\newcommand{\Area}{\operatorname{Area}}

\newcommand{\md}{\operatorname{md}}

\newcommand{\sys}{\operatorname{sys}}

\newcommand{\jac}{{\mathbf J}}
\newcommand{\ap}{\operatorname{ap}}
\newcommand{\apmd}{\ap\md}

\newcommand{\R}{\mathbb{R}}

\begin{document}

%\page{plain}

%\bibliography{abstract}

%\tableofcontents
\pagebreak

% \bibliography{halpha-abbrv}

%\bibliographystyle{alpha}
%\pagenumbering{roman}

\title{The Plateau-Douglas problem for singular configurations and in general metric spaces}

%\subjclass[2010]{49Q05, 53A05, 53C23}

\author{Paul Creutz}

	\address{Department of Mathematics and Computer Science, University of Cologne, Weyertal 86-90, 50931 K\"oln, Germany.}
\email{pcreutz@math.uni-koeln.de}

\author{Martin Fitzi}

\address
  {Department of Mathematics\\ University of Fribourg\\ Chemin du Mus\'ee 23\\ 1700 Fribourg, Switzerland}
\email{martin.fitzi@unifr.ch}

\date{\today}
%\keywords{harmonic maps, quasi-harmonic maps, quadratic isoperimetric inequality, Sobolev maps, metric spaces, H\"older regularity}

\thanks{The first author was partially supported by the DFG grant SPP 2026.
The second author was partially supported by the Swiss National Science Foundation Grants 165848 and 182423. This work was also partially supported by the grant 346300 for IMPAN from the Simons Foundation and the matching 2015-2019 Polish MNiSW fund}

\begin{abstract}
Assume you are given a finite configuration $\Gamma$ of disjoint rectifiable Jordan curves in~$\mathbb{R}^n$. The Plateau-Douglas problem asks whether there exists a minimizer of area among all compact surfaces of genus at most $p$ which span~$\Gamma$. While the solution to this problem is well-known, the classical approaches break down if one allows for singular configurations~$\Gamma$ where the curves are potentially non-disjoint or self-intersecting. Our main result solves the Plateau-Douglas problem for such potentially singular configurations. Moreover, our proof works not only in $\mathbb{R}^n$ but in general proper metric spaces. Thus we are also able to extend previously known existence results of J\"{u}rgen Jost as well as of the second author together with Stefan Wenger for regular configurations. In particular, existence is new for disjoint configurations of Jordan curves in general complete Riemannian manifolds. A minimal surface of \emph{fixed} genus $p$ bounding a given configuration $\Gamma$ need not always exist, even in the most regular settings. Concerning this problem, we also generalize the approach for singular configurations via minimal sequences satisfying conditions of cohesion and adhesion to the setting of metric spaces.
\end{abstract}

\maketitle

\renewcommand{\theequation}{\arabic{section}.\arabic{equation}}
\pagenumbering{arabic}

\section{Introduction and statement of main results}\label{sec:Intro}

\subsection{Introduction}\label{sec:intro-basics}
The classical Plateau problem asked whether any given rectifiable Jordan curve~$\Gamma$ in~$\R^n$ bounds a Sobolev disc of least area. The positive answer was obtained independently by Douglas and Rad\'{o} in the early 1930's,~\cite{Rad30,Dou31}. Over the years their result was generalized from $\mathbb{R}^n$ to so-called homogeneously regular Riemannian manifolds,  metric spaces satisfying curvature bounds in the sense of Alexandrov and particular classes of homogeneously regular Finsler manifolds,~\cite{Mor48,Nik79,MZ10,OvdM14,PvdM17}. The solution of Plateau's problem in proper metric spaces given by Lytchak-Wenger in~\cite{LW15-Plateau} covers all these settings. However, even in $\mathbb{R}^n$, the arguments break down if~$\Gamma$ is allowed to self-intersect. Still the generality of~\cite{LW15-Plateau} and a simple extension trick allowed the first author to solve the Plateau problem for possibly self-intersecting curves in proper metric spaces which satisfy a local quadratic isoperimetric inequality,~\cite{Creutz-singular}. In $\mathbb{R}^n$ this improved a previous existence result due to Hass,~\cite{Has91}.\par
The Plateau-Douglas problem is a variation of the Plateau problem, where one allows for various boundary components and surfaces of nontrivial topology. One way to state the solution obtained by Douglas in~\cite{Dou39} is the following: assume you are given a finite configuration of disjoint rectifiable Jordan curves~$\Gamma$ in~$\mathbb{R}^n$ and a natural number $p\geq 0$. Then there exists an area minimizer among all compact surfaces which have genus at most~$p$ and span~$\Gamma$. Douglas' result has since been extended by Jost to homogeneously regular Riemannian manifolds and recently even further by the second author together with Stefan Wenger to proper metric spaces admitting a local quadratic isoperimetric inequality,~\cite{Jos85,FW-Plateau-Douglas}. Again, the machinery fails if one allows for singular, possibly non-disjoint or self-intersecting configurations. Our main result, Theorem~\ref{thm:main} below, solves the Plateau-Douglas problem for such possibly singular configurations and in general proper metric spaces. The solution for singular configurations is new even in~$\mathbb{R}^n$. Theorem~\ref{thm:main} also generalizes the main results of~\cite{FW-Plateau-Douglas} and~\cite{Creutz-singular} as we are able to drop the assumption that~$X$ admits a local quadratic isoperimetric inequality. In particular, existence is new for regular configurations in complete Riemannian manifolds which might not be homogeneously regular. It is not surprising that existence in this case is harder to obtain, since already for such a setting discontinuous solutions can only be excluded under additional geometric assumptions, cf.~\cite{Mor48}.\par 
Note that the somewhat more modern approach to Plateau's problem via currents as in~\cite{FF60,AK00} does not allow for bounding the topology of solutions, and for singular configurations currents would consider the boundary curves rather as unparametrized objects and could not keep track of the order in which they are traversed, in contrast to our approach. Moreover, beyond the Riemannian setting, there is no appropriate regularity theory available.
\subsection{Main result}
Simple examples show that, without additional assumptions, one cannot hope for reasonably regular area minimizers of prescribed topological type to bound a given contour~$\Gamma$. For example, a Jordan curve in $\R^n$ which is convex and contained in a plane does not span a minimal surface of genus~$p>0$, see~\cite{Mee81}. There are two ways to handle this issue. As in~\cite{Dou39,Jos85} we will state our result in terms of the so-called Douglas condition. It is however not hard to see that that this formulation, which we discuss below, is equivalent to the one via (possibly disconnected) surfaces of bounded topology promoted in Section~\ref{sec:intro-basics}, cf.~\cite{FW-Plateau-Douglas}.

For the convenience of a reader who might not be familiar with the theory of metric space valued Sobolev maps, we first state our main result in the smooth context before moving to the more general setting. To this end, let $X$ be a smooth complete Riemannian manifold and $M$ be a smooth, orientable, compact surface (which might be disconnected). Assume furthermore that all connected components of $M$ have nonempty boundary. For a map $u$ in the Sobolev space $W^{1,2}(M,X)$ we denote by $\Area(u)$ the parametrized Riemannian area of~$u$. \par
Assume now that $M$ has $k\geq 1$ boundary components $\partial M_i$ and $\Gamma$ is a collection of $k$ rectifiable closed curves $\Gamma_j$ in~$X$. By a \textit{rectifiable closed curve} we mean an equivalence class of parametrized rectifiable curves $\gamma\colon S^1\to X$. We identify two such parametrized curves if they are reparametrizations of each other, meaning more precisely that their constant speed parametrizations agree up to a homeomorphism of~$S^1$. We say that a map $u\in W^{1,2}(M,X)$ \textit{spans} $\Gamma$ if for each curve $\Gamma_j$ there exists a boundary component $\partial M_i$ such that the trace $u|_{\partial M_i}$ is a parametrization of~$\Gamma_j$. Let $\Lambda(M,\Gamma,X)$ be the family of Sobolev maps $u\in W^{1,2}(M,X)$ which span~$\Gamma$. We define $$a( M,\Gamma, X):= \inf\{\Area(u): u\in\Lambda( M, \Gamma, X)\}$$ 
and $a_p(\Gamma,X):=a(M,\Gamma,X)$ if $M$ is the (up to a diffeomorphism) unique connected surface of genus $p$ with $k$ boundary components. We say that the \emph{Douglas condition} holds for $p$, $\Gamma$ and $X$ if $a_p(\Gamma,X)$ is finite and
\begin{equation}\label{eq:cond-Douglas}
a_p(\Gamma,X)<a(M,\Gamma,X)
\end{equation}
for every $M$ as in the previous paragraph and of one of the following types. Either $M$ is connected and of genus strictly smaller than $p$, or $M$ is disconnected and of total genus at most $p$. Note that in the case where $\Gamma$ is a single curve and $p=0$, which corresponds to the classical Plateau problem, the Douglas condition is equivalent to the assumption that there is at least one Sobolev disc spanning~$\Gamma$.
\begin{thm}
\label{thm:main0}
Let $X$ be a smooth complete Riemannian manifold and $\Gamma$ a configuration of $k\geq 1$ rectifiable closed curves. Let $M$ be a compact, connected and orientable surface with $k$ boundary components and of genus $p\geq 0$. If the Douglas condition holds for $p, \Gamma$ and $X$, then there exists $u\in \Lambda(M,\Gamma,X)$ as well as a Riemannian metric $g$ on $M$ such that
$$
    \Area(u)=a_p(\Gamma,X)
$$
and $u$ is weakly conformal with respect to~$g$ on $M \setminus u^{-1}(\Gamma)$. Furthermore, if...
\begin{enumerate}[label=(\roman*)]
    \item \label{main0-i}
    ... $X$ is homogeneously regular, then $u$ may be chosen H\"{o}lder continuous on $M$ and smooth on $M \setminus u^{-1}(\Gamma)$.
    \item \label{main0-ii}
    ... $X$ is homogeneously regular and $\Gamma$ is $C^2$, then $u$ may be chosen locally Lipschitz on $M\setminus \partial M$.
        \item \label{main0-iii}
    ... $\Gamma$ is a union of disjoint Jordan curves, then $u$ and $g$ may be chosen such that $u$ is weakly conformal with respect to $g$ on $M$.
\end{enumerate}
\end{thm}
Here, by weakly conformal we mean that almost everywhere the weak differential of~$u$ either vanishes or is angle preserving. 
 Already the most simple example of a figure eight curve in~$\mathbb{R}^2$ shows that self-intersecting curves need not always bound globally weakly conformal area minimizing discs, cf.~\cite{Has91}. So the assumption of \ref{main0-iii} seems quite sharp. Note that the existence of globally H\"{o}lder continuous area minimizers guaranteed by \ref{main0-i} is new already for topologically regular configurations in~$\mathbb{R}^n$ which potentially are of low analytic regularity. Compare the respective discussion for the Plateau problem in~\cite{Creutz-singular}. Without geometric assumptions one cannot hope for the conclusion of \ref{main0-i} to be true. See \cite[p.~809]{Mor48} for a complete Riemannian manifold $X$ and a Jordan curve $\Gamma\subset X$ which only bounds discontinuous area minimizers. Parts \ref{main0-i} and \ref{main0-ii}, respectively \ref{main0-ii} and \ref{main0-iii}, are compatible in the sense that when both respective assumptions are satisfied then one can achieve the conclusion simultaneously for a single map~$u$, compare Remark~\ref{rem:hoelder+lip}. However, if both the assumptions in \ref{main0-i} and \ref{main0-iii} hold, we can only cook up a single area minimizer which is simultaneously weakly conformal and globally H\"{o}lder continuous in the previously known case where all the curves of~$\Gamma$ satisfy a chord-arc condition.

We sketch the main ideas entering in the proof of Theorem~\ref{thm:main}. For \ref{main0-i}, the procedure is conceptually similar to the respective disc type result obtained in~\cite{Creutz-singular}. Namely, we attach a cylinder to each of the curves in~$\Gamma$. This way we obtain a metric space $X_\Gamma$, which admits a local quadratic isoperimetric inequality and contains $X$ isometrically, as well as a regular configuration $\tilde{\Gamma}\subset X_\Gamma$. Now we apply \cite{FW-Plateau-Douglas} to solve the Plateau-Douglas problem for the new pair $(X_\Gamma,\tilde{\Gamma})$ and project the obtained solution down to~$X$. This gives the desired solution for~$(X,\Gamma)$. For~\ref{main0-ii}, the proof follows essentially the same lines. However, the construction is now performed in a way that is more sensitive to the concrete geometric situation. The construction scheme, which is a generalization of the funnel extensions introduced by Stadler in~\cite{Sta:ar}, allows us to obtain an extension space $\hat{X}_\Gamma$ which admits a local quadratic isoperimetric inequality \textit{and} is locally of curvature bounded above in the sense of Alexandrov. This latter feature allows to apply the regularity theory for harmonic maps into spaces of curvature bounded above as developed e.g. in~\cite{KS93,Ser95,BFHMSZ18}, and hence derive the desired Lipschitz regularity. For the special case \ref{main0-iii}, we use $\varepsilon$-thickenings as introduced in~\cite{Wen08-sharp} to approximate $X$ by metric spaces $(X_n)_{n\in \mathbb{N}}$ which admit local quadratic isoperimetric inequalities and contain~$X$ isometrically. Then we apply again~\cite{FW-Plateau-Douglas} to obtain solutions $(u_n)_{n\in \mathbb{N}}$ for the pairs $(X_n,\Gamma)$ respectively. A variant of the Rellich-Kondrachov compactness theorem allows us to pass to a limit surface in~$X$ which is our desired solution. The proof of the remaining general case involves a mix of the arguments discussed for \ref{main0-i} and \ref{main0-iii}.

At this point, we would like to emphasize the following remarkable feature of Theorem~\ref{thm:main0} and its proof: despite major additional complications that arise, the results and methods developed in~\cite{FW-Plateau-Douglas} for the Plateau-Douglas problem in metric spaces are in principle adaptations of respective ones developed for the classical Plateau-Douglas problem in smooth ambient spaces. However, the flexibility of the metric setting therein allows us to draw new conclusions in the smooth setting that seem out of reach within the classical methods.\par 
A theory of metric space valued Sobolev maps has been developed over the last 30 years. With this language at hand, one can generalize all the introduced terminology to the setting where $X$ is a complete metric space, see Sections~\ref{sec:Prelim} and~\ref{sec:regcase} below. Recall that a metric space $X$ is called \emph{proper} if all closed and bounded subsets of~$X$ are compact. In fact, Theorem~\ref{thm:main0} is a special case of the following very general result.
\bt
\label{thm:main}
Let $X$ be a proper metric space and $\Gamma$ a configuration of $k\geq 1$ rectifiable closed curves. Let $M$ be a compact, connected and orientable surface with $k$ boundary components and of genus $p\geq 0$. If the Douglas condition holds for $p$, $\Gamma$ and $X$, then there exists $u\in \Lambda(M,\Gamma,X)$ as well as a Riemannian metric $g$ on $M$ such that $$\Area(u)=a_p(\Gamma,X)$$ 
and $u$ is infinitesimally isotropic with respect to $g$ on $M \setminus u^{-1}(\Gamma)$. Furthermore, if...
\begin{enumerate}[label=(\roman*)]
    \item \label{main-i}
    ... $X$ admits a local quadratic isoperimetric inequality, then $u$ may be chosen H\"{o}lder continuous on $M$ and to satisfy Lusin's property (N).
    \item \label{main-ii}
    ... $X$ is geodesic, admits a local quadratic isoperimetric inequality and is locally of curvature bounded above, and $\Gamma$ is of finite total curvature, then $u$ may be chosen locally Lipschitz on $M\setminus \partial M$.
    \item \label{main-iii}
    ... $\Gamma$ is a union of disjoint Jordan curves, then $u$ and $g$ may be chosen such that $u$ is infinitesimally isotropic with respect to $g$ on $M$.
    \end{enumerate}
\et
The respective assumptions and conclusions in Theorem~\ref{thm:main} are natural metric generalizations of the respective smooth ones in Theorem~\ref{thm:main0}. For example homogeneously regular Riemannian manifolds admit a local quadratic isoperimetric inequality. In fact, the huge class of metric spaces admitting a local quadratic isoperimetric inequality includes also homogeneously regular Finsler manifolds, $\textnormal{CAT}(\kappa)$ spaces, compact Alexandrov spaces as well as more exotic examples such as higher dimensional Heisenberg groups, cf.~\cite{LW15-Plateau}. In particular, the assumption on~$X$ in Theorem~\ref{thm:main}.\ref{main-ii} is satisfied if $X$ is a $\textnormal{CAT}(\kappa)$ space.

We would also like to remark that, despite the fact that we exclusively restrict our discussion to the parametrized Hausdorff area (see Definition~\ref{def:param-Hsdff-area}), an appropriate variant of Theorem~\ref{thm:main} holds for any area functional which induces quasi-convex $2$-volume densities in the sense of \cite{LW17-en-area,APT04} such as the Holmes-Thompson area functional. In order to obtain the respective results, only minor modifications in the proof of the theorem are needed. 
\subsection{Conditions of cohesion and adhesion}
As discussed above, in general one cannot hope for a given configuration~$\Gamma$ of disjoint Jordan curves to bound a minimal surface of prescribed topological type if the Douglas condition for $p$, $\Gamma$ and $X$ fails. However, there are still situations where the Douglas condition fails but one can show the existence of such a desired surface. Namely, if the area infimum may be approximated by a sequence of surfaces which satisfies a geometric nondegeneracy condition, called \emph{condition of cohesion}. In increasingly more general settings this has been shown to hold true in~\cite{Cou37,Shi39,TT88,FW-Plateau-Douglas}. Additional difficulties arise if one allows for singular configurations~$\Gamma$. Imposing an additional so-called \emph{condition of adhesion}, Iseri was able to show a statement of similar spirit for singular configurations in~$\mathbb{R}^n$,~\cite{Ise96}. In Section~\ref{sec:adh} we generalize the definition of adhesion and Iseri's result to the setting of metric spaces. For regular configurations in sufficiently nice ambient spaces, the Douglas condition implies the condition of cohesion for any sequence of surfaces approaching the energy infimum. Note however that nothing similar is true for singular configurations and the condition of adhesion. Hence these results can only be applied to obtain existence for very particular configurations, cf.~\cite{Ise96}.
\subsection{Organization}\label{sec:org} 
 After recalling some basic notions in section~\ref{sec:Prelim}, we discuss the proof of Theorem~\ref{thm:main}.\ref{main-i} in Section~\ref{sec:regcase}, where we first recall some terminology and the main result of \cite{FW-Plateau-Douglas} in Subsection~\ref{sec:recall} before giving the actual proof of (i) in Subsection~\ref{sec:proofreg}. Moving forward, we discuss a generalization of the Cartan-Hadamard theorem due to Bowditch and a gluing result due to Stadler in Subsection~\ref{sec:glob}, and the proof of Theorem~\ref{thm:main}.\ref{main-ii} is performed in Subsection~\ref{sec:fincurv}. Section~\ref{sec:gencase} is then dedicated to the proofs of Theorems~\ref{thm:main} and~\ref{thm:main0} in the general case. In Subsection~\ref{sec:appseq}, we first discuss how general proper metric spaces~$X$ can be approximated by more regular spaces admitting local quadratic isoperimetric inequalities and when one can pass from a sequence of fillings within the approximating spaces to a limit filling in~$X$. Then in Subsection~\ref{sec:reductionsfillings}, we recall two devices from~\cite{FW-Plateau-Douglas} that allow, in spaces admitting a local quadratic isoperimetric inequality, to lower the topological type of an area minimizing sequence whenever this sequence degenerates. These devices are combined in Section~\ref{sec:reductionsapp} with the approximating spaces discussed before. The proof of Theorem~\ref{thm:main} is then completed in Section~\ref{sec:proofgen}. In Section~\ref{sec:translation} we briefly discuss how Theorem~\ref{thm:main0} follows from Theorem~\ref{thm:main}. Finally in Section~\ref{sec:adh}, we discuss the method using minimizing sequences satisfying conditions of cohesion and adhesion.
\subsection*{Acknowledgements}
We wish to thank our respective PhD advisors Alexander Lytchak and Stefan Wenger for their great support. Furthermore, advancements in this project were made during a research visit to the Institute of Mathematics of the Polish Academy of Sciences in 2019 and we would like to thank the IMPAN for its hospitality.
\section{Preliminaries}\label{sec:Prelim}
\subsection{Basic notation}\label{sec:prelim-notations}
We write $|v|$ for the Euclidean norm of a vector $v\in\R^2$, $$D:=\{z\in\R^2:|z|<1\}$$ for the open unit disc in $\R^2$ and $\bar{D}$ for its closure.
The differential at $z$ of a (weakly) differentiable map $\varphi$ between smooth manifolds is denoted $D \varphi_z$.

For a subset $A\subset\R^2$, $|A|$ denotes its Lebesgue measure. If $(X,d)$ is a metric space then we use the notation $\hm_{X}^2(A)$ for the $2$--dimensional Hausdorff measure of a subset $A\subset X$. The normalizing constant is chosen such that $\hm_X^2$ coincides with the $2$--dimensional Lebesgue measure when $X$ is Euclidean $\R^2$. Thus, the Hausdorff $2$--measure $\hm_g^2:=\hm_{(M,g)}^2$ on a 2--dimensional Riemannian manifold $(M,g)$ coincides with the Riemannian area. 
\subsection{Seminorms}\label{sec-energy-area-isotropy}
The \emph{(Reshetnyak) energy} of a seminorm $s$ on $\R^2$ is defined by  $$\mathbf{I}_+^2(s):= \max\{s(v)^2: v\in\R^2, |v|=1\}.$$
If $s$ is a norm on $\R^2$, then the \emph{Jacobian} of $s$ is defined as the unique number $\jac(s)$ satisfying 
$$\hm^2_{(\R^2,s)}(A)=\jac(s)\cdot|A|$$
for some and thus every subset $A\subset \R^2$ such that $|A|>0$. For a degenerate seminorm $s$ we set $\jac(s):=0$. A seminorm $s$ on $\R^2$ is \emph{isotropic} if $s=0$ or if it is a norm and the ellipse of maximal area contained in $\{v\in\R^2: s(v)\leq 1\}$ is a Euclidean ball. If $s$ is a Euclidean seminorm, i.e. if $s$ is induced by a (potentially degenerate) inner product, then $s$ is isotropic precisely if it is a scalar multiple of the standard Euclidean norm~$|\cdot |$.

If $s$ is a seminorm on a $2$-dimensional Euclidean vector space~$V$ then we define the concepts of Jacobian, energy, and isotropy by identifying $V$ with Euclidean $(\R^2,|\cdot|)$ via a linear isometry.
\subsection{Metric space valued Sobolev maps}\label{sec:sobmaps}
Let $(X, d)$ be a proper metric space and let $M$ be a smooth, compact, orientable $2$--dimen\-sional manifold, possibly disconnected and with non-empty boundary.
We fix a Riemannian metric $g$ on $M$ and let  $\Omega\subset M$ be an open set. 
\bd
A measurable $u\colon \Omega\to X$ belongs to the Sobolev space $W^{1,2}(\Omega, X)$ if there exists $h\in L^2(\Omega)$ with the following property. For every real-valued $1$--Lipschitz function $f$ on $X$ the composition $f\circ u$ belongs to the classical Sobolev space $H^{1,2}(\Omega\setminus \partial M)$  and $$|D (f\circ u)_z|_g\leq h(z)$$ for almost every $z\in \Omega$.
\ed

If $u\in W^{1,2}(\Omega, X)$ then for almost every $z\in \Omega$ there exists a seminorm $\apmd u_z$ on $T_zM$, called \emph{approximate metric derivative}, such that
$$\ap\lim_{v \to 0} \ \ \frac{d(u(\exp_z(v)), u(z)) - \apmd u_z (v)}{|v|_g} = 0,$$
where the approximate limit is taken within $T_z M$ and $\exp_z$ denotes the exponential map of~$g$ at~$z$. See \cite{EG15} for the definition of approximate limits.

Assume $N=(N,h)$ is a smooth complete Riemannian manifolds. Then, by Nash's theorem, there is an isometric embedding $\iota\colon N\to \mathbb{R}^m$ (in the Riemannian sense). Equivalently one may define $W^{1,2}(\Omega,N)$ as the set of measurable mappings $u\colon\Omega\to N$ such that $\iota \circ u$ lies in the classical Sobolev space $H^{1,2}(\Omega\setminus \partial M,\mathbb{R}^m)$; compare e.g. Lemma~9.3.3 and Exercise~2 in Section~9 of~\cite{Jos17}. In particular, for every Sobolev map $u\in W^{1,2}(\Omega,N)$ there is a measurable weak differential $Du\colon T\Omega\to TN\subset N\times \mathbb{R}^m$. At almost every $z\in \Omega$ the approximate metric derivative is given by
\begin{equation}
\label{eq:apmd-weakdiff}
\apmd u_z(v)=|D u_z(v)|_h  \ \ \textnormal{for all } v \in T_z \Omega,    
\end{equation}
compare Theorem~6.4 and the subsequent remark in~\cite{EG15}.

The approximate metric derivative allows one to define the Reshetnyak energy and the parametrized Hausdorff area of a Sobolev map using the pointwise quantities introduced in Section~\ref{sec-energy-area-isotropy} above.
\bd
 The (Reshetnyak) energy of $u\in W^{1,2}(\Omega, X)$ with respect to~$g$ is defined by $$E_+^2(u, g):= \int_{\Omega} \mathbf{I}_+^2(\apmd u_z)\;d\hm^2_{g}(z).$$
\ed
The energy $E^2_+$ is conformally invariant in the sense that $$E^2_+(u\circ \varphi,g')=E^2_+(u,g) $$ whenever $\varphi\colon (M',g')\to (M,g)$ is a conformal diffeomorphism.
\bd\label{def:param-Hsdff-area}
The parametrized (Hausdorff) area of $u\in W^{1,2}(\Omega,X)$ is defined by
$$\Area(u):=\int_\Omega \jac(\apmd u_z)\; d\hm^2_g(z).$$
If $A\subset \Omega$ is measurable, then the area of the restriction $u|_A$ is defined analogously.
\ed
It is easy to see that $$\Area(u\circ \varphi)=\Area(u)$$ for any biLipschitz homeomorphism $\varphi\colon \Omega'\to \Omega$. In particular, $\Area(u)$ is independent of the choice of the Riemannian metric~$g$. A measurable map $u\colon\Omega \to X$ satisfies \emph{Lusin's property (N)} if $\mathcal{H}^2_X(u(A))=0$ for every null set $A\subset \Omega$. If $u\in W^{1,2}(\Omega,X)$, then by the area formula
$$\Area(u)\leq \int_X \# \{z\in \Omega: u(z)=x\}\; d\hm^2_X(x),$$
with equality if $u$ satisfies Lusin's property (N); see~\cite{Kar07}.
\bd\label{def:inf-isotropic}
 A map $u\in W^{1,2}(M, X)$ is infinitesimally isotropic with respect to the metric~$g$ on a measurable subset $A\subset M$ if for almost every $z\in A$ the approximate metric derivative $\apmd u_z$ is isotropic with respect to $g(z)$. If no subset $A\subset M$ is specified, it is understood that $u$ is infinitesimally isotropic with respect to $g$ on $M$.
\ed
It is not hard to see that $$\label{eq:ae}\Area(u)\leq E^2_+(u,g),$$
where equality holds precisely if $u$ is infinitesimally isotropic and the approximate metric derivative of $u$ at almost every $z\in M$ is a Euclidean seminorm, compare~\cite{LW17-en-area}.

If $\Omega\subset M\setminus \partial M$ is a Lipschitz domain, then for every $u\in W^{1,2}(\Omega,X)$ there is a well defined \emph{trace} $\trace(u)\in L^{2}(\partial \Omega,X)$. %{\color{red}We also denote the trace by $u|_{\partial \Omega}$. Indeed,} 
If $u$ extends to a continuous map $\bar{u}$ on~$\bar{\Omega}$, then the trace is simply given by $\bar{u}|_{\partial \Omega}$.  Hence, in abuse of notation, we also denote the trace of $u$ by $u|_{\partial \Omega}$. If no continuous extension exists, define $\trace(u)$ locally around $p\in \partial \Omega$ in the following way. Choose an open neighborhood $U$ of $p$ and a biLip\-schitz map $\psi\colon (0,1)\times [0,1)\to M$ such that $\psi((0,1)\times (0,1)) = U\cap \Omega$ and $\psi((0,1)\times\{0\}) = U\cap \partial\Omega$. Then for almost every $s\in (0,1)$ the trace at $\psi(s,0)$ is given by $\lim_{t\searrow 0} (u\circ\psi)(s,t)$, compare \cite{KS93}.
\section{Proof for regular metric spaces}
\label{sec:regcase}
\subsection{The Plateau-Douglas problem for regular configurations}
\label{sec:recall}
Let $\mathcal{M}(k)$ be the family of compact, orientable, smooth surfaces $M$ with $k$ boundary components and such that each connected component of $M$ has non-empty boundary. Denote by~$M_{k,p}$ the, up to a diffeomorphism, unique connected surface in $\mathcal{M}(k)$ of genus~$p$. A \textit{reduction} of $M_{k,p}$ is a surface $M^*\in\mathcal{M}(k)$ with one of the following properties. Either $M^*$ is connected and has genus at most $p-1$ or $M^*$ has several connected components and the total genus of $M^*$ is at most $p$. Since the Euler characteristic of $M_{k,p}$ is given by $$\chi(M_{k,p})=2-2p-k,$$ it follows that $\chi(M^*)>\chi(M_{k,p})$ for any reduction $M^*$ of $M_{k,p}$, and hence $\chi(M^*)=k$ if and only if $M^*$ is the union of $k$ smooth discs. For $M\in\mathcal{M}(k)$ with $n>1$ connected components, we say that $M^*$ is a reduction of $M$ if there exists a partition $M^*=M_1^*\cup...\cup M_n^*$ such that each $M_l^*$ is the reduction of exactly one connected component of $M$. Notice that for any $M\in\mathcal{M}(k)$ there are only finitely many reductions $M^*$ up to diffeomorphism, and that any reduction $M^{**}$ of such $M^*$ is also a reduction of $M$.

 Let $\Gamma=\bigcup \Gamma_j$ be a configuration of $k\geq 1$ rectifiable closed curves in a complete metric space $X$ and $p\geq 0$. By defining 
 $$a_p^*(\Gamma,X):=\min\{a(M^*,\Gamma,X):M^*\text{ is a reduction of }M_{k,p}\},$$
 the Douglas condition~(\ref{eq:cond-Douglas}) can be rewritten as
 $$a_p(\Gamma,X)<a_p^*(\Gamma,X).$$
 We would like to point out that the notion of reduction used here is broader than the one given in \cite{FW-Plateau-Douglas}, where a reduction of the second type consists of \textit{exactly} two connected components. Consequently, the Douglas condition used in \cite{FW-Plateau-Douglas} is \`{a} priori a weaker assumption than the respective one in this article, which turns out to be more convenient for us. However, the two conditions are in fact equivalent.
 This follows since $a_p(\Gamma,X)<\infty$ implies that all curves $\Gamma_j$ lie in the same component of rectifiable connectedness of $X$, i.e. the curves can be joined pairwise by paths of finite length,
  and using this fact one can show that $a(M^*,\Gamma,X)\leq a(M^{**},\Gamma,X)$ whenever $M^{**}$ is a reduction of a reduction $M^*$ of~$M_{k,p}$.

The basis for our proof of Theorem~\ref{thm:main} in the special cases~\ref{main-i} and \ref{main-ii} will be the existence results \cite[Theorem 1.2]{FW-Plateau-Douglas} and \cite[Theorem 1.4.(iii)]{FW-Plateau-Douglas} for Jordan curves, which we now state as a combined theorem for convenience of the reader.

\bt\label{thm:main-QII-Jordan}
Let $X$ be a proper metric space admitting a local quadratic isoperimetric inequality, $\Gamma\subset X$ the disjoint union of $k\geq 1$ rectifiable Jordan curves and $p\geq 0$. If the Douglas condition~(\ref{eq:cond-Douglas}) holds for $p$, $\Gamma$ and $X$, then there exists a continuous $u\in\Lambda(M_{k,p},\Gamma,X)$ and a Riemannian metric $g$ on $M_{k,p}$ such that $$\Area(u)=a_p(\Gamma,X)$$ and $u$ is infinitesimally isotropic with respect to~$g$. Furthermore, if every Jordan curve in $\Gamma$ is chord-arc, then any such $u$ is H\"{o}lder continuous on $M_{k,p}$ and satisfies Lusin's property (N).
\et

Here, a metric space $X$ is said to admit a \textit{$(C,\ell_0)$-quadratic isoperimetric inequality} if every closed Lipschitz curve $c\colon S^1\to X$ of length $\ell(c)\leq\ell_0$ is the trace of a Sobolev disc $u\in W^{1,2}(D,X)$ satisfying $$\Area(u)\leq C\cdot\ell(c)^2.$$ If there is no need to specify the constants $C,\ell_0>0$, we simply say that $X$ admits a \textit{local quadratic isoperimetric inequality}. A Jordan curve $\Gamma$ is called \textit{chord-arc} if it is biLipschitz equivalent to $S^1$.

The following replacement lemma will be used in the proof of Lemma~\ref{lem:fill-cyl-atmost}. It follows from the proof of \cite[Lemma 4.8]{LW-intrinsic} and the gluing result \cite[Theorem~1.12.3]{KS93}. While \cite[Lemma 4.8]{LW-intrinsic} is stated for disc-type surfaces, the arguments in the proof thereof are local around the boundary curve and can be applied without changes to the present situation.
\bl\label{lem:reparam}
Let $X$ be a complete metric space admitting a local quadratic isoperimetric inequality, $\Gamma\subset X$ a configuration of $k\geq 1$ rectifiable closed curves and $M\in\mathcal{M}(k)$. Then for every $u\in \Lambda(M,\Gamma,X)$ and $\varepsilon >0$ there is $v\in \Lambda(M,\Gamma,X)$ such that
$$\Area(v)\leq\Area(u)+\varepsilon$$ and the continuous representative of $\trace(v)|_{\partial M_i}$ is a constant speed parametrization for each $i\in\{1,\dots,k\}$.
\el

Lemma~3.2 is applied in the proofs of Propositions~5.1 and~6.1 in~\cite{FW-Plateau-Douglas}. It is one of the implications in~\cite{FW-Plateau-Douglas} making use of the assumption of a local quadratic isoperimetric inequality. In fact the only implications needing this assumption and used in the proof of the existence result therein may be phrased as Lemmas~\ref{lem:reduc-sys-bound} and~\ref{lem:reduc-tr-equicont} below. While these lemmas seem to heavily rely on the assumption, it is an open question whether Lemma~\ref{lem:reparam}, which enters in their proofs, holds true without it or not.

\subsection{Proof of Theorem~\ref{thm:main}.\ref{main-i}}
\label{sec:proofreg}
Let $X$ be a complete metric space and $\Gamma$ a configuration of $k\geq 1$ rectifiable closed curves $\Gamma_j$ in $X$. Since the Douglas condition fails as soon as $k>1$ and one of the curves $\Gamma_j$ is constant, and since the minimization problem is trivial for a single constant curve $\Gamma$, we  may assume without loss of generality that $\Gamma_1,\dots ,\Gamma_k$ are all nonconstant. For each $j$, let $S_j$ be a geodesic circle of circumference $\ell(\Gamma_j)$, let $\gamma_j\colon S_j\to X$ be a unit speed parametrization of $\Gamma_j$ and $Z_j:=S_j\times [0,1]$ be the cylinder equipped with the product metric. We define the quotient space $X_\Gamma$ as the disjoint union $X\sqcup Z_1\sqcup\dots\sqcup Z_k$ under the identification $\gamma_j(p)\sim (p,0)$ for every $p\in Z_j$, and we equip this space with the quotient metric,~see for example~\cite{BrH99}. Furthermore, let $P_\Gamma\colon X_\Gamma\to X$ be the projection given by
$$ P_\Gamma(x):=\begin{cases}x & x\in X, \\ \gamma_j(p) & x=(p,t)\in Z_j.\end{cases}$$  The proof of \cite[Lemma 4.1]{Creutz-singular} shows that $X\subset X_\Gamma$ isometrically and $P_\Gamma\colon X_\Gamma\to X$ is a $1$-Lipschitz retraction. Lastly, we define $\Tilde{\Gamma}_j$ as the (equivalence class of the) rectifiable curve $p\mapsto (p,1)\in Z_j,\, p\in S_j$, and $\Tilde{\Gamma}$ as the configuration consisting of the curves $\Tilde{\Gamma}_1,\dots, \Tilde{\Gamma}_k$. Then $\Tilde{\Gamma}$ is a configuration of disjoint chord-arc curves and $P_\Gamma \circ \Tilde{\Gamma}_j=\Gamma_j$ for each~$j$.

\bl \label{lem:fill-cyl-atleast}
Let $X$ be a complete metric space, $\Gamma\subset X$ a configuration of $k\geq 1$ rectifiable closed curves and $M\in\mathcal{M}(k)$. Then for every $u\in\Lambda(M,\Tilde{\Gamma},X_\Gamma)$ one has $P_\Gamma \circ u \in \Lambda(M,\Gamma,X)$ and
$$\Area(u)\geq\Area(P_\Gamma\circ u)+\sum_{j=1}^k\hm{\color{red}}^2(Z_j).$$
In particular, one has the inequality
$$a(M, \Tilde{\Gamma}, X_\Gamma)\geq a(M, \Gamma, X) + \sum_{j=1}^k\hm^2(Z_j).$$
\el

\begin{proof}
Let $u\in \Lambda(M,\Tilde{\Gamma},X_\Gamma)$. Without loss of generality, we may assume that $M$ is connected. By the 1-Lipschitz continuity of $P_\Gamma$, we have that $P_\Gamma\circ u\in \Lambda(M,\Gamma,X)$. Since $P_\Gamma(Z_j)$ is contained in the rectifiable curve $\Gamma_j$, the area formula in Section~\ref{sec:sobmaps} implies that
$$\Area\left((P_\Gamma\circ u)|_{u^{-1}(Z_j)}\right)=0.$$
Thus, since the restriction $P_\Gamma|_X$ is an isometry, we obtain
\[
\Area(u)=\Area(u|_{u^{-1}(X)})+\sum_j\Area\left(u|_{u^{-1}(Z_j)}\right)=\Area(P_\Gamma \circ u)+\sum_j\Area\left(u|_{u^{-1}(Z_j)}\right).
\]
To complete the proof, it therefore suffices to show that
\begin{equation}
\label{eqi}
\Area\left(u|_{u^{-1}(Z_j)}\right)\geq \mathcal{H}^2(Z_j)
\end{equation}
for each $j$. In order to see this, fix $j$ and define $Y_j$ as the quotient space $X_\Gamma / A$, where $A:=X \cup \bigcup_{i\neq j} Z_i$. Then $Y_j$ is isometric to $Z_j/(S_j\times \{0\})$. Hence $Y_j$ is homeomorphic to $\bar{D}$ and, by \cite[Theorem 3.2]{Cre20a}, admits a local quadratic isoperimetric inequality. Furthermore, let $Q_j\colon X_\Gamma\to Y_j$ be the 1-Lipschitz map given by $Q_j(x):=[x]$. Then the composition $Q_j\circ u$ is an element in $\Lambda(M,Q_j\circ \Tilde{\Gamma}, Y_j)$ with 
\begin{equation}\label{eqi1}
\Area(Q_j\circ u)=\Area\left(u|_{u^{-1}(Z_j)}\right).\end{equation}
Let $\partial M_i$ be the boundary component of $M$ such that $\trace(u)|_{\partial M_i}$ is an element of $\Gamma_j$, and consider $M$ embedded into a smooth compact surface $\Tilde{M}\in\mathcal{M}(1)$ of same genus as that of $M$ such that each boundary component $\partial M_l$ bounds a topological disc in $\Tilde{M}$ except for $\partial M_i$, which agrees with the boundary component of $\Tilde{M}$. The map $Q_j\circ u$ extends naturally onto $\Tilde{M}$ by setting its value on $\Tilde{M}\setminus M$ to be $[x]$ for any $x\in X$, yielding a map $v_j\in \Lambda(\Tilde{M}, Q_j\circ \Tilde{\Gamma}_j, Y_j)$ satisfying \begin{equation}\label{eqi2}\Area(v_j)=\Area(Q_j\circ u).\end{equation}
 Apparently, there exists a surface $M^*$, either being equal to $\Tilde{M}$ or else being a reduction of it, such that $$a(\Tilde{M},Q_j\circ\Tilde{\Gamma}_j, Y_j)=a(M^*,Q_j\circ\Tilde{\Gamma}_j,Y_j)$$ and the Douglas condition holds for $M^*,Q_j\circ\Tilde{\Gamma}_j$ and $Y_j$. Hence by Theorem~\ref{thm:main-QII-Jordan} there exists a continuous map $w_j\in \Lambda(M^*, Q_j\circ\Tilde{\Gamma}_j, Y_j)$ 
 satisfying Lusin's property~(N) and \begin{equation}\label{eqi3}\Area(w_j)\leq \Area(v_j).\end{equation}Since $Y_j$ is homeomorphic to $\bar{D}$ with boundary curve $Q_j\circ\Tilde{\Gamma}_j$, it follows that $w_j$ is surjective. Otherwise assume $p\in Y_j\setminus w_j(M^*)$. Then $Q_j\circ\tilde{\Gamma}_j$, considered as a $1$-cycle, would be a generator of $H_1(Y_j\setminus \{p\})\cong H_1(\bar{D}\setminus \{0\})\cong \mathbb{Z}$ and at the same time would bound the $2$-chain defined in $Y_j\setminus \{p\}$ by $w_j$, which is a clear contradiction. Hence, by the area formula, we have \begin{equation}\label{eqi4}
\Area(w_j)=\int_{Y_j} \#\left\{w_j^{-1}(x)\right\}\;d\hm{\color{red}}^2(x)\geq \hm^2{\color{red}}(Y_j)=\hm^2{\color{red}}(Z_j).\end{equation}
Combining \eqref{eqi1}, \eqref{eqi2}, \eqref{eqi3} and~\eqref{eqi4}, we finally obtain~\eqref{eqi}.
\end{proof}
 While we did not need to assume a local quadratic isoperimetric inequality on~$X$ in the previous lemma, this assumption is required  in the proof of the upcoming reverse inequality.
\bl \label{lem:fill-cyl-atmost}
Let $X$ be a complete metric space admitting a local quadratic isoperimetric inequality, $\Gamma\subset X$ a configuration of $k\geq 1$ rectifiable closed curves and $M\in\mathcal{M}(k)$. Then one has
$$a(M, \Tilde{\Gamma}, X_\Gamma)\leq a(M, \Gamma, X) + \sum_{i=1}^k \hm^2(Z_j).$$
\el
\begin{proof}
Let $\varepsilon>0$. By Lemma~\ref{lem:reparam} there exists $v\in \Lambda(M,\Gamma,X)$ such that
$$\Area(v)\leq a(M,\Gamma,X)+\varepsilon$$
and such that $\trace(v)|_{\partial M_i}$ is a constant speed parametrization for each $i$. We relabel the boundary components of $M$ such that $\trace(v)|_{\partial M_j}$ is an element of $\Gamma_j$ for each~$j$. Embed $M$ diffeomorphically into a smooth compact surface $\Tilde{M}\in\mathcal{M}(k)$ such that $\Tilde{M}\setminus \operatorname{int}(M)$ is the disjoint union of $k$ smooth cylinders $\Omega_j$ with boundary, each $\Omega_j$ having $\partial M_j$ as one boundary component. Notice that $\Tilde{M}$ is diffeomorphic to $M$. Now if $\tilde{\gamma}_j\colon S_j \to X_\Gamma$ is a constant speed parametrization of $\Tilde{\Gamma}_j$, then the inclusion $\iota_j\colon Z_j\to X_\Gamma$ is a Lipschitz homotopy between $\Tilde{\gamma}_j$ and $\gamma_j$ of area $\hm^2(Z_j)$. Thus, by identifying $\Omega_j$ with $Z_j$ via a biLipschitz homeomorphism, there exist maps $w_j\in W^{1,2}(\Omega_j,X_\Gamma)$ with trace $\Tilde{\gamma}_j$ respectively $ \gamma_j=\trace(v)|_{\partial M_j}$ and of area $\hm^2(Z_j)$. Let $w\colon \Tilde{M}\to X_\Gamma$ be the mapping obtained by stitching $v$ together with every $w_j$ along $\partial M_j$, which is a well-defined element in $W^{1,2}(\Tilde{M},X_\Gamma)=W^{1,2}(M,X_\Gamma)$ by \cite[Thm. 1.12.3]{KS93}. Then $w$ spans $\Tilde{\Gamma}$ and satisfies $$a(M,\Tilde{\Gamma},X_\Gamma)\leq \Area(w)=\Area(v)+\sum_{j=1}^k\Area(w_j)\leq a(M,\Gamma,X)+\sum_{j=1}^k \hm^2(Z_j)+\varepsilon.$$ Since $\varepsilon>0$ was chosen arbitrary, the assertion in the lemma follows and the proof is complete.
\end{proof}
With these preparations at hand, it is now not hard to give a proof of Theorem~\ref{thm:main}.\ref{main-i}.
\begin{proof}[Proof of Theorem~\ref{thm:main}.\ref{main-i}]
Since $X$ admits a local quadratic isoperimetric inequality, it follows from the proof of~\cite[Theorem~3.2]{Cre20a} that $X_\Gamma$ admits a local quadratic isoperimetric inequality as well. Lemma~\ref{lem:fill-cyl-atleast} together with Lemma~\ref{lem:fill-cyl-atmost} imply that one has the equality
\begin{equation}
\label{eqfeq}
    a(\Tilde{M},\Tilde{\Gamma},X_\Gamma)=a(\Tilde{M},\Gamma,X)+\sum_{j=1}^k \hm^2(Z_j)
    \end{equation}
for every $\Tilde{M}\in \mathcal{M}(k)$. Hence the Douglas condition $$a_p(\Tilde{\Gamma},X_\Gamma)<a_p^*(\Tilde{\Gamma},X_\Gamma)$$ holds for $p$, $\Tilde{\Gamma}$ and $X_\Gamma$. Since $\Tilde{\Gamma}$ is a disjoint configuration of chord-arc curves, we have by Theorem~\ref{thm:main-QII-Jordan} that there is a H\"{o}lder continuous $v\in\Lambda(M,\Tilde{\Gamma},X_\Gamma)$ satisfying Lusin's property (N) and a Riemannian metric $g$ on $M$ such that
$$\Area(v)=a_p(\Tilde{\Gamma},X_\Gamma)$$ and $v$ is infinitesimally isotropic with respect to~$g$. By Lemma~\ref{lem:fill-cyl-atleast} and equation~\eqref{eqfeq} the projection $u:=P_\Gamma\circ v\in\Lambda(M,\Gamma,X)$ then satisfies $$\Area(u)=a_p(\Gamma,X).$$ Moreover, since $P_\Gamma$ is isometric on $X$, the map $u$ is infinitesimally isotropic with respect to~$g$ on $M\setminus u^{-1}(\Gamma)\subset M\setminus v^{-1}(X_\Gamma \setminus X)$. Thus the proof of \ref{main-i} is complete.
\end{proof}
\section{Interior Lipschitz regularity}
%\label{sec:fincurv}
\subsection{Upper curvature bounds}
\label{sec:glob}
Let $X$ be a metric space. Closed piecewise geodesic curves in $X$ will be denoted $\overline{x_0x_1\dots x_m}$, where $x_i\in X$ indicate the endpoints of the geodesic segments. For $\kappa\in \mathbb{R}$, let $D_\kappa$ be the diameter of the model space~$M^2_\kappa$ of constant curvature~$\kappa$. That is, $D_\kappa=\pi/\sqrt{\kappa}$ for $\kappa >0$ and $D_\kappa=\infty$ for $\kappa\leq 0$. A geodesic triangle $\overline{x y z}$ will be called \emph{$\kappa$-admissible} if $\ell(\overline{xyz})<2D_\kappa$. For every $\kappa$-admissible triangle $\overline{x y z}$, there is a (up to isometry) unique comparison triangle $\overline{ x_\kappa y_\kappa z_\kappa}$ in $M^2_\kappa$ which has the same side lengths. A $\kappa$-admissible triangle $\overline{x y z}$ is called \emph{CAT($\kappa$)} if there is a $1$-Lipschitz map $f\colon \overline{ x_\kappa y_\kappa z_\kappa}\to \overline{x y z}$ such that $f(x_\kappa)=x$, $f(y_\kappa)=y$ and $f(z_\kappa)=z$. We say that $X$ is a \emph{CAT($\kappa$) space} if $X$ is geodesic and every $\kappa$-admissible triangle in~$X$ is $\textnormal{CAT}(\kappa)$, and call $X$ \emph{locally CAT($\kappa$)} if every point in $X$ has a neighbourhood which is a $\Cat(\kappa)$ space. Two standard facts are that $\Cat(\kappa)$ spaces are also $\Cat(\kappa')$ for any $\kappa'\geq \kappa$, and that balls of radius at most $D_\kappa/2$ in $\Cat(\kappa)$ spaces are themselves $\Cat(\kappa)$ spaces. Finally, we say that $X$ is \emph{locally of curvature bounded above} if every point $p\in X$ has a neighbourhood $U_p$ which is a $\Cat(\kappa_p)$ space for some $\kappa_p\in \mathbb{R}$. By the preceeding observations, we may always assume that $\kappa_p>0$ and $U_p$ is a small ball.

If $X$ is geodesic and locally $\Cat(0)$, then the Cartan-Hadamard theorem states that $X$ is a $\Cat(0)$ space if and only if $X$ is simply connected. Aiming to handle also spaces satisfying positive upper curvature bounds, we discuss a variant of this result due to Bowditch. For Lipschitz curves $\gamma_0,\gamma_1\colon S^1\to A\subset X $, we say that $\gamma_0$ is \emph{monotonically homotopic to $\gamma_1$ in $A$} if there exists a continuous homotopy $h\colon[0,1]\times S^1 \to A$ such that $h(0,\cdot)= \gamma_0$, $h(1,\cdot)=\gamma_1$ and $\ell(h(t,\cdot))\leq \ell(\gamma_0)$ for all $t\in [0,1]$. We say that $\gamma$ is \emph{monotonically nullhomotopic in $A$} if $\gamma$ is monotonically homotopic to a constant curve in $A$. If $X$ is a $\Cat(\kappa)$ space, then Reshetnyak's majorization theorem (see for example \cite{AKP19}) implies that every closed Lipschitz curve in~$X$ of length smaller than~$2D_\kappa$ is monotonically nullhomotopic. Dually, the following holds by Theorem~$3.1.2$ in~\cite{Bow95}.
\begin{thm}
\label{thm:bow}
Let $X$ be a proper geodesic metric space, $\kappa \in \mathbb{R}$ and $A\subset X$ be compact such that the $D_\kappa$-neighbourhood of $A$ is locally $\Cat(\kappa)$. If a $\kappa$-admissible triangle $\Delta\subset A$ is monotonically nullhomotopic in~$A$, then $\Delta$ is $\Cat(\kappa)$.
\end{thm}
Theorem~3.1.2 in \cite{Bow95} is stated under the assumption that the entire space $X$ is locally $\Cat(\kappa)$. However, as discussed in Section~3.6 of~\cite{Bow95}, the argument is local in the $D_\kappa$-neighbourhood of any set in which $\Delta$ is monotonically nullhomotopic, and hence the proof readily gives Theorem~\ref{thm:bow}. As a corollary of Theorem~\ref{thm:bow}, we obtain the following result allowing to derive quantitatively controlled "local globalizations". 
\begin{cor}
\label{cor:CH}
Let $X$ be a proper geodesic metric space, $\kappa\in \mathbb{R}$ and $B(p,r)\subset X$ a ball which is locally $\Cat(\kappa)$. If every triangle $\Delta\subset \bar{B}(p,r/2)$ is monotonically nullhomotopic in $\bar{B}(p,r/2)$, then $\bar{B}(p,\bar{r})$ is a $\Cat(\bar{\kappa})$ space, where $\bar{\kappa}=\bar{\kappa}(\kappa,r)$ and $\bar{r}=\bar{r}(\kappa,r)$ only depend on $\kappa$ and $r$.
\end{cor}
\begin{proof}
Set $\bar{\kappa}:=\max \{\kappa,4\pi^2r^{-2}\}$ and $\bar{r}:=D_{\bar{\kappa}}/4$. Note that $\bar{\kappa}$ is chosen such that $D_{\bar{\kappa}}\leq r/2$. To see that $\bar{B}(p,\bar{r})$ is convex, let $x,y\in \bar{B}(p,\bar{r})$ and observe that any geodesic triangle $\overline{pxy}$ is $\bar{\kappa}$-admissible and contained in $\bar{B}(p,2\bar{r})\subset \bar{B}(p,r/2)$, and hence by assumption monotonically nullhomotopic within~$\bar{B}(p,r/2)$. Then Theorem~\ref{thm:bow} implies that $\overline{pxy}$ is $\textnormal{CAT}(\bar{\kappa})$. Since $\bar{r}< D_{\bar{\kappa}}/2$, it follows that $\overline{pxy}\subset \bar{B}(p,\bar{r})$, and we conclude that $\bar{B}(p,\bar{r})$ is convex. Now let $\overline{xyz}\subset \bar{B}(p,\bar{r})$. Then $\overline{xyz}$ is $\bar{\kappa}$-admissible and monotonically nullhomotopic in $\bar{B}(p,r/2)$. Again Theorem~\ref{thm:bow} implies that $\overline{xyz}$ is $\textnormal{CAT}(\bar{\kappa})$.
\end{proof}
For $\alpha\geq 0$ and $r >0$, we let $S_{\alpha,r}$ be the ball of radius~$r$ around the vertex in the cone over a compact interval of length~$\alpha$  (see \cite{BBI01} for the definition of cones), and call $S_{\alpha,r}$ the \emph{sector} of radius~$r$ and angle~$\alpha$. On any sector, we fix an orientation so that the left leg and the right leg of $S_{\alpha,r}$ are defined. The following lemma generalizes \cite[Lemma 21]{Sta:ar} to spaces satisfying positive upper curvature bounds.
\begin{lem}
\label{lem:secglue}
Let $\kappa\geq 0$, $0<r\leq D_\kappa/2$, $X$ be a proper $\Cat(\kappa)$ space, $p\in X$ and $\eta_1,\dots,\eta_l,\nu_1,\dots, \nu_l\subset X$ geodesic segments all of length $r$ and starting at~$p$. For $i=1,\dots,l$, let $\alpha_i\in [0,\pi]$ be the angle at $p$ between $\eta_i$ and $\nu_i$, and let $S_i$ be the sector of angle $2\pi-\alpha_i$  and radius~$r$. Then the space $Z$, obtained by gluing each sector $S_i$ to $X$ via isometric identifications of its left leg with $\eta_i$ and its right leg with~$\nu_i$, is a $\Cat(\kappa)$ space.
\end{lem}
 In the lemma, the isometric identifications are chosen such that~$p$ corresponds to the vertex point in~$S_i$. In the following, we assume without further mentioning that the orientations of isometric identifcations are chosen in such a natural way.
\begin{proof} 
By induction, it is sufficient to prove the statement for $l=1$, and hence we set $\eta:=\eta_1$, $\nu:=\nu_1$ and $\alpha:=\alpha_1$. Reshetnyak's gluing theorem (see for example~\cite{BrH99}) implies that the space~$Y$, obtained by gluing~$S_{\pi-\alpha,r}$ to~$X$ via an isometric identification of the left leg of~$S_{\pi-\alpha,r}$ and~$\eta$, is a $\Cat(\kappa)$ space. Observe that the angle in~$Y$ between the right leg~$\eta'$ of~$S_{\pi-\alpha,r}$ and~$\nu$ equals~$\pi$ and that the length of the concatenation $\eta'\cup \nu$ is at most~$D_k$. Hence the curve $\eta'\cup \nu$ is a geodesic in~$Y$ and in particular a convex subset of~$Y$, see \cite[Proposition 1.7]{BrH99}. Thus the claim follows from another application of Reshetnyak's theorem upon noting that $Z$ may be constructed alternatively by gluing the sector $S_{\pi,r}$ to $Y$ via isometric identifications of its left leg with $\eta'$ and its right leg with $\nu$.
\end{proof}
\subsection{Proof of Theorem~\ref{thm:main}.\ref{main-ii}}
\label{sec:fincurv}
Let $X$ be a metric space which is locally of curvature bounded above. The \emph{total curvature} of a closed piecewise geodesic curve $\overline{x_0 x_1 \dots x_m}$ in~$X$ is defined by
$$
\sigma(\overline{x_0x_1\dots x_m}):=\sum_{i=0}^m (\pi- \beta_i),
$$
where $\beta_i$ denotes the angle at $x_i$ between the geodesic segments $\overline{x_i x_{i-1}}$ and $\overline{x_i x_{i+1}}$. Let $L$ be a closed rectifiable curve. The curve $\overline{x_0 x_1 \dots x_m}$ is called \emph{inscribed to~$L$} if the points $x_0,x_1,\dots, x_m$ lie on $L$ and are traversed by $L$ in cyclic order. The \emph{total curvature of~$L$}, denoted~$\sigma(L)$, may be defined as $\lim_{n\to \infty}  \sigma(L_n)$, where $(L_n)$ is a sequence of closed piecewise geodesic curves which are inscribed to~$L$ and converge uniformly to~$L$, see \cite[Proposition~2.4]{ML03}.
\begin{proof}[Proof of Theorem~\ref{thm:main}.\ref{main-ii}]
Let $X$ be as in the statement of the theorem. Assume first $L=\overline{x_0x_1\dots x_m}$ is a closed piecewise geodesic curve in $X$. For $i=0,\dots,m$, we set $S_i:=S_{\pi-\beta_i,1}$ and $Q_i:=I_i\times [0,1]$, where $I_i\subset \R$ is a compact interval of length~$d(x_i,x_{i+1})$. We define a geodesic metric cylinder~$\hat{Z}_L$ by gluing the left end interval of each~$Q_i$ isometrically to the right leg of $S_i$ and the right end interval of each~$Q_i$ to the left leg of~$S_{i+1}$. Then, by Reshetnyak's gluing theorem, balls of radius at most $\ell(L)/4$ in $\hat{Z}_L$ are $\Cat(0)$ spaces. Denote the inner boundary curve of~$\hat{Z}_L$ by~$\bar{L}$ and the outer boundary curve of $\hat{Z}_L$ by $\hat{L}$. There exist a $1$-Lipschitz retraction $\hat{P}_L\colon \hat{Z}_L\to \bar{L}$ such that $\hat{P}_L \circ \hat{L}=\bar{L}$, as well as a $(\ell(L)+\sigma(L))$-Lipschitz homotopy $h_L\colon S^1\times [0,1]\to \hat{Z}_L$ between $\bar{L}$ and $\hat{L}$ such that $\Area(h)=\mathcal{H}^2(\hat{Z}_L)$. In particular, $\bar{L}$ is a geodesic circle of circumference~$\ell(L)$ and there is a canonical unit-speed parametrization $c_L\colon \bar{L}\to L$. Now let $L$ be any closed rectifiable curve of finite total curvature. All the properties discussed for piecewise geodesic curve are quantitative and hence stable under ultralimits; see e.g.~\cite{AKP19} for the definition and properties of ultralimits. Thus we may approximate $L$ by a sequence $(L_n)$ of $L$-inscribed piecewise geodesic curves, perform the construction for each $L_n$, pass to an ultralimit and obtain that there exist $\hat{Z}_{L}$, $\hat{L}$, $\bar{L}$, $c_L$, $h_{L}$, $\hat{P}_{L}$ as above, all enjoying the very same properties.

Let $\hat{Z}_j:=\hat{Z}_{\Gamma_j}$ for $j=1,\dots,k$. We define the quotient space $\hat{X}_\Gamma$ as the disjoint union $X\sqcup \hat{Z}_1 \sqcup \dots \sqcup \hat{Z}_k$ under the identification $c_{\Gamma_j}(p)\sim p$ for $p\in \bar{\Gamma}_j$, and we equip this space with the quotient metric. Also, we let $\hat{P}_\Gamma\colon\hat{X}_\Gamma\to X$ be the $1$-Lipschitz retraction given by $\hat{P}_\Gamma(x):=x$ for $x\in X$ and $\hat{P}_\Gamma (x)=\hat{P}_{\Gamma_j}(x)$ for $x\in \hat{Z}_j$. By Reshetnyak's majorization theorem each $\hat{Z}_j$ admits a local quadratic isoperimetric inequality. This, together with the facts that $\hat{P}_\Gamma$ is $1$-Lipschitz and $X$ admits a local quadratic isoperimetric inequality, makes it straight forward to modify the proof of \cite[Theorem 3.2]{Cre20a} and derive that the space $\hat{X}_\Gamma$ admits a local quadratic isoperimetric inequality. Let $\hat{\Gamma}$ be the configuration formed by $\hat{\Gamma}_1,...,\hat{\Gamma}_k$. The properties discussed above allow us to imitate the proofs of Lemmas~\ref{lem:fill-cyl-atleast} and~\ref{lem:fill-cyl-atmost} for the configuration $\hat{\Gamma}\subset \hat{X}_\Gamma$, and hence derive that
\begin{equation}
\label{eq:catarea}
a(\Tilde{M},\hat{\Gamma},\hat{X}_\Gamma)=a(\Tilde{M},\Gamma,X)+\sum_{i=1}^k\mathcal{H}^2(\hat{Z}_i)
\end{equation}
for every $\Tilde{M}\in\mathcal{M}(k)$.

So far we have not achieved any advantage from our more complicated construction over the one in Section~\ref{sec:proofreg}. However, and this is the crucial difference, now we claim that $\hat{X}_\Gamma$ is locally of curvature bounded above. Since $\hat{X}_\Gamma\setminus X$ is locally $\Cat(0)$, it suffices to show that every $p \in X$ has a $\Cat$ neighbourhood within $\hat{X}_\Gamma$. So let $p$ in $X$ and choose $\kappa>0$ as well as $0<r<D_\kappa/2$ such that $B_X(p,r)$ is a $\Cat(\kappa)$ space. The proof that $X$ is locally of curvature bounded above will be completed by showing that $\bar{B}_{\hat{X}_\Gamma}(p,\bar{r})$ is a $\Cat(\bar{\kappa})$ space, where $\bar{\kappa}$ and $\bar{r}$ are as in the statement of Corollary~\ref{cor:CH}. Since $\bar{\kappa}$ and $\bar{r}$ are independent of $\Gamma$ and the $\Cat(\bar{\kappa})$ condition is stable under ultralimits, we lose no generality in assuming that $\Gamma_1,\dots,\Gamma_k$ are piecewise geodesic curves. Thus it remains to verify the assumptions of Corollary~\ref{cor:CH}. Clearly, $B_{\hat{X}_\Gamma}(p,r)\setminus \Gamma$ is locally $\Cat(\kappa)$. Since we assumed $\Gamma$ consists of piecewise geodesic curves, for $q\in B_{\hat{X}_\Gamma}(p,r)\cap \Gamma$ and $s>0$ sufficiently small the ball $\bar{B}_{\hat{X}_\Gamma}(q,s)$ is obtained from $\bar{B}_X(q,s)$ as the space $Z$ is obtained from $X$ in Lemma~\ref{lem:secglue}. Thus the lemma states that $\bar{B}_{\hat{X}_\Gamma}(q,s)$ is a $\Cat(\kappa)$ space and hence we conclude that $B_{\hat{X}_\Gamma}(p,r)$ is locally $\Cat(\kappa)$. To verify the other assumption of Corollary~\ref{cor:CH}, let $\Delta\subset \bar{B}_{\hat{X}_\Gamma}(p,r/2)$ be a geodesic triangle. Sliding $\Delta$ down to $X$ we see that $\Delta$ is monotonically homotopic in~$\bar{B}_{\hat{X}_\Gamma}(p,r/2)$ to a curve $\eta\subset X$. Since $\bar{B}_X(p,r/2)$ is a $\Cat(\kappa)$ space and $\ell(\eta)<2 D_\kappa$, Reshetnyak's majorization theorem implies in turn that $\eta$ is monotonically nullhomotopic in~$\bar{B}_X(p,r/2)$. Hence we may apply Corollary~\ref{cor:CH} and conclude the claim.

Departing from \eqref{eq:catarea} and the fact that $\hat{X}_\Gamma$ admits a local quadratic isoperimetric inequality, we can proceed as we did when proving~\ref{main-i} in the last section. The advantage is now that by~\cite{Ser95}, see also~\cite[Theorem~1.3]{BFHMSZ18}, the minimizer $v\in \Lambda(M,\hat{X}_\Gamma,\hat{\Gamma})$ is locally Lipschitz on $M\setminus \partial M$, and hence so is our final solution $u=\hat{P}_\Gamma \circ v$. In order to apply these regularity results, note that $v$ is a continuous harmonic map into a space which is locally of curvature bounded above. Harmonicity of $v$ follows since $v$ is infinitesimally isotropic and $\hat{X}_\Gamma$ is locally of curvature bounded from above and hence has property (ET), see~\cite[Section~11]{LW15-Plateau}.
\end{proof}
\begin{rmrk}
\label{rem:hoelder+lip}
The map~$u$ we produce in the proof of Theorem~\ref{thm:main}.\ref{main-ii} is also globally H\"{o}lder continuous on $M$. This follows as in the proof of Theorem~\ref{thm:main}.\ref{main-i} upon noting that the configuration $\hat{\Gamma}$ we construct consists of chord-arc curves.
\end{rmrk}
\section{General Case}
\label{sec:gencase}
Throughout this section, we use the terminology introduced in the beginning of Section~\ref{sec:regcase}.
\subsection{Approximating sequences}
\label{sec:appseq}
Let $X$ be a complete metric space. We call a metric space~$Y$ an \emph{$\varepsilon$-thickening of} ~$X$ if~$Y$ contains~$X$ isometrically and $X$ is $\varepsilon$-dense in~$Y$. We will need the following variant of the thickening results obtained in~\cite{Wen08-sharp} and \cite{LWYar}.
\bl
\label{prop:thick}
There is a universal constant~$C\geq 0$ such that for every proper metric space $X$ and $\varepsilon>0$, there exists a $(C \varepsilon)$-thickening~$Y$ of~$X$ such that $Y$ is proper and admits a $\left(C,\varepsilon \right)$-quadratic isoperimetric inequality.
\el
If $X$ is geodesic, then Lemma~\ref{prop:thick} follows readily from \cite[Lemma 3.3]{LWYar} and in this case, the space $Y$ may also be chosen geodesic. This version suffices to obtain Theorem~\ref{thm:main} in the special case that $X$ is geodesic, and hence in particular to obtain Theorem~\ref{thm:main0}. Thus for the convenience of a reader who is only interested in Theorem~\ref{thm:main} for geodesic target spaces, the general proof of Lemma~\ref{prop:thick} is postponed to the appendix.

 Let $X$ be a proper metric space and $(Y_n)_{n\in \mathbb{N}}$ a sequence of proper $\varepsilon_n$-thickenings of $X$. We call $(Y_n)$ an \emph{$X$-approximating sequence} if $\varepsilon_n\to 0$. The following consequence of the generalized Rellich-Kondrachov compactness theorem,~\cite[Theorem 1.13]{KS93}, allows to pass from a sequence of maps in approximating spaces to a limit map in~$X$.
\begin{prop}\label{prop:Rell-Kond-comp}
Let $X$ be a proper space and $\Gamma$ be a configuration of $k\geq 1$ disjoint rectifiable Jordan curves in $X$. Let $M\in\mathcal{M}(k)$ be connected and endowed with a Riemannian metric $g$. Assume that there exist an $X$-approximating sequence $(Y_n)_{n\in \mathbb{N}}$ and mappings $u_n\in \Lambda(M,\Gamma,Y_n)$ of uniformly bounded energies $E^2_+(u_n,g)$ and such that the traces $\trace(u_n)\colon\partial M\to \Gamma$ are equicontinuous with respect to~$g$. Then there is $u\in \Lambda(M,\Gamma,X)$ such that
\begin{equation}
\label{eqae}
    \Area(u)\leq \limsup_{n\to \infty} \Area(u_n)\quad\& \quad  E^2_+(u,g)\leq \limsup_{n\to \infty} E^2_+(u_n,g).
\end{equation}
\end{prop}
The proof is the following standard argument, which is similar to respective steps e.g. in the proofs of~\cite[Theorem 1.5]{GWar} and~\cite[Theorem 5.1]{LWYar}.
\begin{proof}
Let $Z$ be the proper metric space obtained by gluing all the spaces $Y_n$ along~$X$. Note that $Y_n\subset Z$ isometrically and hence $\Lambda(M,\Gamma,Y_n)\subset \Lambda(M,\Gamma,Z)$ for each $n\in \mathbb{N}$. For fixed $p\in \Gamma$,  \cite[Lemma 2.4]{FW-Plateau-Douglas} implies that there is a constant~$C$ such that
\begin{equation*}
\int_M d^2(p,u_n(z)) \, d\hm^2_g(z)\leq C \cdot \left(\textnormal{diam}(\Gamma)^2+E^2_+(u_n,g)\right)
\end{equation*}
for all $n\in \mathbb{N}$. In particular,
$$\sup_{n\in \N} \left[ \int_M d^2(p,u_n(z)) \, d\hm^2_g(z) + E^2_+(u_n,g) \right]<\infty.$$
Thus by the metric space version of the Rellich-Kondrachov compactness theorem,~\cite[Theorem 1.13]{KS93}, there is $v\in W^{1,2}(M,Z)$ such that $v_j\to v$ in $L^2(M,Z)$. In fact, since $(Y_n)_{n\in \mathbb{N}}$ is an approximating sequence, we may assume that $v$ takes values in $X\subset Z$ and hence $v\in W^{1,2}(M,X)$. By lower semicontinuity of area and energy, see e.g.~\cite{LW15-Plateau}, the inequalities~\eqref{eqae} are satisfied for $u$. Finally, the Arzel\`{a}-Ascoli theorem and \cite[Theorem 1.12.2]{KS93} imply that $v\in \Lambda(M,\Gamma,X)$. 
\end{proof}
\subsection{Reductions of fillings}
\label{sec:reductionsfillings} 
Let $X$ be a complete metric space, $p\geq 0$ and $\Gamma\subset X$ a configuration of $k\geq 1$ disjoint rectifiable Jordan curves $\Gamma_j$. The two following results are needed for the proof of Lemma~\ref{lem:alt} and can be extracted from the proofs of \cite[Proposition 6.1]{FW-Plateau-Douglas} and \cite[Proposition 5.1]{FW-Plateau-Douglas} respectively. For the first lemma, we assume that $k+p>2$, which is equivalent to the assumption that the surface $M_{k,p}$ is neither of disc- nor of cylindrical type. In this case $M_{k,p}$ may be endowed with a \emph{hyperbolic} metric, which we define to be a Riemannian metric~$g$ of constant sectional curvature $-1$ and such that the boundary $\partial M_{k,p}$ is geodesic with respect to~$g$. By a \textit{relative geodesic} in $(M_{k,p},g)$ we mean either a simple closed geodesic in $M_{k,p}$ or a geodesic arc with endpoints on $\partial M_{k,p}$ that is non-contractible via a homotopy of curves of the same type. We define $\sys_{\rm rel}(M_{k,p},g)$ as the infimal length of relative geodesics in~$(M_{k,p},g)$. Furthermore, we choose for each $\rho>0$ a parameter $\rho'_\Gamma=\rho'_\Gamma(\rho)$ as in the first paragraph in the proof of \cite[Proposition~6.1]{FW-Plateau-Douglas}. That is, for each $\rho>0$ we choose $0<\rho'_{\Gamma}< \rho$ such that whenever two points $x,x'\in\Gamma$ satisfy $d_X(x,x')\leq \rho'_\Gamma$, then they lie on the same Jordan curve~$\Gamma_j$ and the shorter segment of $\Gamma_j$ between $x$ and $x'$ has length at most $\rho$. The notation emphasizes that $\rho'_\Gamma$ only depends on the induced metric on $\Gamma\subset X$.
\bl \label{lem:reduc-sys-bound}
Let $C,K,\rho>0$. Assume $X$ admits a $(C,2\rho)$-quadratic isoperimetric inequality and $g$ is a hyperbolic metric on $M_{k,p}$ such that $$\sys_{\rm rel}(M_{k,p},g)<\min\left\{\frac{\rho'^2_{\Gamma}(\rho)}{4K}, \arsinh\left(\frac{1}{\sinh(2)}\right)\right\}.$$
Then for every $u\in\Lambda(M_{k,p},\Gamma,X)$ with $E_+^2(u,g)\leq K$, there exist a reduction~$M^*$ of~$M_{k,p}$ and a map $u^*\in\Lambda(M^*, \Gamma, Y)$ such that
$$\Area(u^*)\leq\Area(u)+8C\rho^2.$$
\el
An analogue of the above lemma holds for cylindrical $M_{k,p}$ endowed with a \textit{flat} metric, which we define as a Riemannian metric with vanishing sectional curvature and such that the Riemannian area of $(M_{k,p},g)$ is equal to 1 and the boundary $\partial M_{k,p}$ geodesic. The analogue follows by using a basic flat collar (instead of a hyperbolic one) in the proof of \cite[Proposition 6.1]{FW-Plateau-Douglas}. Compare also the respective remark in the proof of \cite[Theorem 1.2]{FW-Plateau-Douglas}.

For the second lemma, we assume that $k+p\geq 2$, hence we only exclude that $M_{k,p}$ is of disc-type. Let $g$ be a Riemannian metric on $M_{k,p}$ and $0<\delta_g<1$ be so small that every point $z_0\in\partial M_{k,p}$ has a neighbourhood in $(M_{k,p},g)$ which is the image of the set \[B:=\{z\in\C:|z|\leq 1\text{ and }|z-1|<\sqrt{\delta_g}\}\]under a 2-biLipschitz diffeomorphism $\psi$ with $z_0=\psi(1)$.
\bl \label{lem:reduc-tr-equicont}
Let $C,K,\rho>0$. Assume that $X$ admits a  $(C,2\rho)$-quadratic isoperimetric inequality and $0<\delta\leq\delta_g$ is so small that
$$\pi\cdot\left(\frac{8K}{|\log(\delta)|}\right)^{\frac{1}{2}}<\rho'_\Gamma(\rho).$$
If there exist $u\in\Lambda(M_{k,p}, \Gamma, Y)$ with $E_+^2(u,g)\leq K$ and a subarc $\gamma^-\subset\partial M_{k,p}$ satisfying
$$\ell_g(\gamma^-)\leq \delta\quad\&\quad\ell_X(\trace(u)\circ \gamma^-)>\rho,$$ then there exist a reduction $M^*$ of $M_{k,p}$ and a map $u^*\in\Lambda(M^*, \Gamma, X)$ such that
$$\Area(u^*)\leq\Area(u)+8C\rho^2.$$
\el
\subsection{Reductions of approximating sequences}
\label{sec:reductionsapp}
Let $X$ be a proper metric space and $\Gamma$ be a configuration of $k\geq 1$ disjoint rectifiable Jordan curves in $X$ and $p\geq 0$. The next proposition is going to be important in the proof of Theorem~\ref{thm:main}. 
\bp\label{prop:mainstep}
Let $(Y_n)$ be an $X$-approximating sequence. If there exist maps $u_n\in \Lambda(M_{k,p},\Gamma,Y_n)$ satisfying
$$ a:=\limsup_{n\to \infty} \Area(u_n)<a_p^*(\Gamma,X),$$
then there exists $u\in \Lambda(M_{k,p},\Gamma,X)$ such that $\Area(u)\leq a$. Moreover, for any sequence $(g_n)$ of Riemannian metrics on $M_{k,p}$, there exists $u$ as above and a Riemannian metric $g$ on $M_{k,p}$ such that
$$E_+^2(u,g)\leq \limsup_{n\to\infty} E_+^2(u_n,g_n).$$
\ep
The proposition follows by repeatedly applying the next lemma.
\bl
\label{lem:alt}
Let $(Y_n)$ be an $X$-approximating sequence, $M\in\mathcal{M}(k)$, $(g_n)$ be a sequence of Riemannian metrics on $M$ and $u_n\in \Lambda(M,\Gamma,Y_n)$ be fillings such that $\Area(u_n)$ is uniformly bounded. Then one of the following two options holds. Either there is $u\in \Lambda(M,\Gamma,X)$ and a Riemannian metric~$g$ on $M$ such that
 $$\Area(u)\leq \limsup_{n\to\infty}\Area(u_n)\quad\&\quad E_+^2(u,g)\leq\limsup_{n\to\infty}E_+^2(u_n,g_n),$$ or there exist a reduction $M^*$ of $M$, an $X$-approximating sequence $(Y_n^*)$ and maps $u_n^*\in\Lambda(M^*,\Gamma,Y_n^*)$ such that
 \begin{equation}
 \label{eq:limsup-reduc}
   \limsup_{n\to\infty}\Area(u_n^*)\leq \limsup_{n\to\infty}\Area(u_n). 
 \end{equation}
\el
\begin{proof}[Proof of Proposition~\ref{prop:mainstep}]
Let $M$, $Y_n$, $u_n$ and $g_n$ be as in the proposition. If the first possibility in Lemma~\ref{lem:alt} when applied to these elements is true, i.e. if the existence of $u\in\Lambda(M,\Gamma,X)$ and a metric $g$ on $M$ as in this lemma is given, then the proposition follows immediately. We claim that the second possibility in the lemma cannot occur. Otherwise, we could iteratedly apply Lemma~\ref{lem:alt} to $M^*$, the sequences $(Y_n^*)$ and $(u_n^*)$ given by the lemma and arbitrarily chosen metrics $g_n^*$ on $M^*$, as well as their respective successors, until eventually the first possibility holds. This has to be the case after finitely many iterations, since the Euler characteristic strictly increases when passing to a reduction, but is also bounded from above by $k$ in our setting. Thus we would obtain a reduction $M^*$ of $M$ and a map $u\in \Lambda(M^*,\Gamma,X)$ such that
$$
\Area(u)\leq \limsup_{n\to \infty} \Area(u_n)<a_p^*(\Gamma,X),
$$
which gives a contradiction.
\end{proof}
At the end of this section, we give a proof for Lemma~\ref{lem:alt}. It is based on Proposition~\ref{prop:Rell-Kond-comp} as well as Lemmas~\ref{lem:reduc-sys-bound} and \ref{lem:reduc-tr-equicont}.
\begin{proof}[Proof of Lemma~\ref{lem:alt}]

Without loss of generality, we may assume that $M$ is connected. Define $$a:=\limsup_{n\to\infty}\Area(u_n)<\infty\quad\text{\&}\quad e:=\limsup_{n\to\infty}E_+^2(u_n,g_n).$$ If $e$ is infinite, we choose a sequence of auxiliary metrics $g_n'$ on $M$ satisfying $$E_+^2(u_n,g_n')\leq \frac{4}{\pi}\Area(u_n)+1,$$ which exist by \cite[Theorem 1.2]{FW-epsConf} and \cite[Section 5]{FW-epsConf}. Thus, after potentially redefining $g_n:=g_n'$, we may assume that $e$ is finite.

We first address the special setting where $\Gamma$ is a single Jordan curve and $M$ a disc-type surface. We may assume that $M=\bar{D}$ and, since all Riemannian metrics on $\bar{D}$ are conformally equivalent, that each $g_n$ is equal to the standard Euclidean metric $g_{\mathrm{Eucl}}$. Now precompose each~$u_n$ with a conformal diffeomorphism~$\varphi_n$ of~$\bar{D}$ such that $\trace(u_n\circ\varphi_n)$ satisfies for each $n$ the same prefixed three-point condition on $\partial D$ and $\Gamma$,~see p. 1149 in~\cite{LW15-Plateau}. Note that the maps $v_n:=u_n\circ\varphi_n$ satisfy $\Area(v_n)=\Area(u_n)$ and $E_+^2(v_n,g_{\mathrm{Eucl}})=E_+^2(u_n,g_{\mathrm{Eucl}})$. %\todo{Beim Definieren von $E_+^2$ sagen, dass Energie konform invariant}.
It then follows by \cite[Proposition 7.4]{LW15-Plateau} that the family $\{\trace(v_n):n\in\N\}$ is equicontinuous, and therefore by Proposition~\ref{prop:Rell-Kond-comp} that there exists $u\in\Lambda(\bar{D}, \Gamma,X)$ with
$$\Area(u)\leq \limsup_{n\to\infty}\Area(v_n)=a\quad\text{\&}\quad E_+^2(u,g_{\mathrm{Eucl}})\leq \limsup_{n\to\infty}E_+^2(v_n,g_{\mathrm{Eucl}})=e$$
as in the first option proposed by the lemma.

From now on, we assume that $M$ is a connected surface which is not of disc-type. Since every conformal class of Riemannian metrics on $M$ has a hyperbolic representative (respectively a flat one if $M$ is of cylindrical type), we lose no generality in assuming that all the metrics $g_n$ are hyperbolic (respectively flat). In the rest of the proof, we discuss three different cases of outcomes in which ultimately either Lemma~\ref{lem:reduc-sys-bound}, Lemma~\ref{lem:reduc-tr-equicont} or Proposition~\ref{prop:Rell-Kond-comp} is used to deduce one of the options stated in the lemma itself.

First assume that
\begin{equation}\label{eq:sys-bound}
    \inf \{\sys_{\rm rel}(M,g_n):n\in\N\}>0.
\end{equation}
Then by \cite[Theorem 3.3]{FW-Plateau-Douglas} (respectively its analogue for flat metrics) there exist diffeomorphisms $\varphi_n$ of $M$ and a metric $g$ on $M$ such that the pullback-metrics $\varphi_n^*g_n$ converge (up to a subsequence) smoothly to $g$. This convergence implies for the maps $v_n:=u_n\circ\varphi_n\in\Lambda(M,\Gamma,Y_n)$ that
$$E_+^2(v_n,g)\leq C_n\cdot E_+^2(u_n,g_n),$$
where $C_n\geq 1$ tends to $1$ as $n\to\infty$. In particular, the energies $E_+^2(v_n,g)$ are uniformly bounded.
Now assume furthermore that the family
\begin{equation}\label{eq:equicont}
    \{\trace(v_n):n\in\N\}\quad\text{is equicontinuous} 
\end{equation}
with respect to the metric $g$. Then by Proposition~\ref{prop:Rell-Kond-comp} there exists $u\in\Lambda(M,\Gamma, X)$ with
$$\Area(u)\leq a\quad\text{\&}\quad E_+^2(u,g)\leq e$$
as in the first option of the lemma.

In the remaining two cases, we discuss the outcomes if either the bound (\ref{eq:sys-bound}) does not hold; or if it does indeed, but property (\ref{eq:equicont}) fails for the traces of the constructed maps $v_n\in\Lambda(M, \Gamma, Y_n)$. Let $$\rho_j:=\frac{1}{\sqrt{C2^{j+3}}},$$%\todo{Vllt hier dann $\sigma_j$ statt $\rho_j$ und $\sigma_j':=\rho'_\Gamma(\sigma_j))$?}
where $C\geq 0$ is the universal constant from Lemma~\ref{prop:thick}, and $\rho_j':=\rho'_\Gamma(\rho_j)$ for each $j\in\N$. We claim that in either of these subcases, there exist a sequence of reductions $M_j^*$ of $M$, a subsequence $(u_{n_j})\subset \Lambda(M, \Gamma, Y_{n_j})$, $(2C\rho_j)$-thickenings $Y_j^*$ of $Y_{n_j}$ and fillings $$u_j^*\in\Lambda(M_j^*, \Gamma, Y_j^*)$$ such that 
$$\Area(u_j^*)\leq \Area(u_{n_j})+2^{-j}.$$
The existence of a sequence as implied in the lemma is then true by the following two observations. Firstly, there are only finitely many reductions of $M$ up to diffeomorphism, hence we may assume that each $M_j^*$ is equal to the same reduction $M^*$ of $M$ by passing to a subsequence of $M_j^*$. Secondly, the spaces $Y_j^*$ are $(\varepsilon_{n_j}+2C\rho_j)$-thickenings of $X$, where $\varepsilon_n$ is the thickening parameter of $Y_n$, and thus $(Y_j^*)$ an $X$-approximating sequence.

We continue by showing the claim and first suppose that (\ref{eq:sys-bound}) is violated. We only discuss the case for hyperbolic metrics, the situation for flat metrics being analogous. The assumption on the systoles of $g_n$ implies that there exists a subsequence $(g_{n_j})$ such that
$$ \sys_{\rm rel}(M,g_{n_j})=:\lambda_j\to 0.$$
Choosing this subsequence appropriately, we may assume that
$$\lambda_j<\min \left\{ \frac{{\rho'_j}^2}{4K},\arsinh\left(\frac{1}{\sinh(2)}\right)\right\},$$
where we define $K:=\sup_n E_+^2(u_n,g_n)<\infty$. By Lemma~\ref{prop:thick}, for each $j$ there exists a $(2C\rho_j)$-thickening $Y_j^*$ of $Y_{n_j}$ admitting a $(C, 2\rho_j)$-quadratic isoperimetric inequality. Since the spaces $Y_j^*$ contain $X$ (and hence $\Gamma$) isometrically and since the metrics $g_n$ are all hyperbolic, we have by Lemma~\ref{lem:reduc-sys-bound} that there exist reductions $M_j^*$ of $M$ and maps $u_{j}^*\in\Lambda(M_j^*,\Gamma, Y_j^*)$ with
$$\Area(u_j^*)\leq \Area(u_{n_j})+8C \rho_j^2\leq \Area(u_{n_j})+2^{-j}.$$
This shows the claim in the first subcase.

Lastly, we address the case where (\ref{eq:sys-bound}) is true, but (\ref{eq:equicont}) is violated for the obtained metric~$g$. Choose for each $j\in\N$ a number $0<\delta_j\leq \delta_g$ such that
$$\pi\cdot\left(\frac{8K}{|\log(\delta_j)|}\right)^{\frac{1}{2}}\leq \rho'_j.$$ From the assumption of nonequicontinuity of $\{\trace(v_n)\}$, it follows that there exists $\varepsilon>0$ such that for every $j$ there exists a map $\trace(v_{n_j})\colon M\to Y_{n_j}$ and a segment $\gamma_j^-\subset\partial M$ satisfying
 $$\ell_{g}(\gamma_j^-)\leq \delta_j\quad\&\quad\ell_X(\trace(v_{n_j})\circ\gamma_j^-)>\varepsilon.$$
Notice that for all $j$ big enough we have that $\rho_j\leq\varepsilon$, so in particular $$\ell_X(\trace(v_{n_j})\circ\gamma_j^-)>\rho_j.$$ 
 Let $Y_j^*$ be given analogously as in the previous subcase. Then by Lemma~\ref{lem:reduc-tr-equicont} there exist reductions $M_j^*$ of $M$ and mappings $u_j^*\in\Lambda(M_j^*,\Gamma, Y_j^*)$ satisfying
$$\Area(u_j^*)\leq\Area(v_{n_j})+8C\rho_j^2\leq\Area(u_{n_j})+2^{-j}.$$
This shows the claim in the second subcase and completes the proof of the lemma.
\end{proof}
\subsection{Proof of the main result}
\label{sec:proofgen}
Finally, we are able to complete the proof of Theorem~\ref{thm:main}.
 \begin{proof}[Proof of Theorem~\ref{thm:main}]
 The statements \ref{main-i} and \ref{main-ii} of the theorem have already been proved in Sections~\ref{sec:proofreg} and \ref{sec:fincurv}. Thus it remains to show \ref{main-iii} as well as existence in the general case, where $X$ might not admit a local quadratic isoperimetric inequality and $\Gamma$ might be a configuration of overlapping or self-intersecting curves.
 
 We begin with the proof of part \ref{main-iii} and assume that $\Gamma$ is a collection of disjoint rectifiable Jordan curves. For $n\in \mathbb{N}$ we set $Y_n:=X$ and choose maps $u_n\in \Lambda(M,\Gamma,X)$ such that
 \[
 \Area(u_n)\leq a_p(\Gamma,X) +2^{-n}.
 \]
 Since we assumed that the Douglas condition holds for $p$, $\Gamma$ and $X$, we may apply Proposition~\ref{prop:mainstep} to the sequences $(Y_n)$ and $(u_n)$. This shows that  $$\Lambda_{\mathrm{min}}:=\{u\in\Lambda(M,\Gamma,X):\Area(u)=a_p(\Gamma,X)\}$$ is nonempty. Choose sequences of maps $u_n\in \Lambda_{\mathrm{min}}$ and Riemannian metrics $g_n$ on $M$ such that $$\lim_{n\to\infty} E_+^2(u_n,g_n)=\inf\{E_+^2(w,h):w\in\Lambda_{\mathrm{min}},\,h\text{ a Riemannian metric on }M\}=:e.$$ Applying Proposition~\ref{prop:mainstep} to the sequences $(Y_n)$, $(g_n)$ and $(u_n)$, one sees that there exist $u\in\Lambda_{\mathrm{min}}$ and a Riemannian metric $g$ on $M$ such that  $E_+^2(u,g)= e.$ Then by \cite[Corollary 1.3]{FW-epsConf} $u$ is infinitesimally isotropic with respect to~$g$. This completes the proof in the special case that the configuration is assumed to consist of disjoint Jordan curves.
 
 We move on to the general case. Let $(X_n)$ be an $X$-approximating sequence, where every $X_n$ admits some local quadratic isoperimetric inequality: such an approximating sequence exists by Proposition~\ref{prop:thick}. % one may find an $X$-approximating sequence $(X_n)$ such that every $X_n$ admits a local quadratic isoperimetric inequality.
 Then $(Y_n):=((X_n)_\Gamma)$ defines an $X_\Gamma$-approximating sequence, where the collar extensions are performed as defined in Section~\ref{sec:proofreg}. By Lemma~\ref{lem:fill-cyl-atmost}, there exist maps $u_n\in \Lambda(M,\Tilde{\Gamma},Y_n)$  such that
 \[
 \Area(u_n)\leq a_p(\Gamma,X_n)+\sum_{j=1}^k\mathcal{H}^2(Z_j)+2^{-n}\leq a_p(\Gamma,X)+\sum_{j=1}^k\mathcal{H}^2(Z_j)+2^{-n}.
 \]
 Then by Lemma~\ref{lem:fill-cyl-atleast}, and since the Douglas condition holds for $p$, $\Gamma$ and $X$, one has
\[
\limsup_{n\to \infty} \Area(u_n)\leq a_p(\Gamma,X)+\sum_{j=1}^k\mathcal{H}^2(Z_j)<a_p^*(\Gamma,X)+\sum_{j=1}^k\mathcal{H}^2(Z_j)\leq a_p^*(\Tilde{\Gamma},X_\Gamma).
\]
Thus applying Proposition~\ref{prop:mainstep} to the sequences $(Y_n)$ and $(u_n)$ shows that the Douglas condition holds for $p$, $\Tilde{\Gamma}$ and $X_\Gamma$ and that 
\[
a_p(\Tilde{\Gamma},X_\Gamma)\leq a_p(\Gamma,X)+\sum_{j=1}^k\mathcal{H}^2(Z_j).
\]
Since $\Tilde{\Gamma}$ is a configuration of disjoint Jordan curves, the Douglas condition and the first part of the proof imply that there exist $v\in\Lambda(M,\Tilde{\Gamma},X_\Gamma)$ and a Riemannian metric $g$ on $M$ such that $\Area(v)=a_p(\Tilde{\Gamma},X_\Gamma)$ and $v$ is infinitesimally isotropic with respect to~$g$. For the projection $u:=P_\Gamma\circ v$ Lemma~\ref{lem:fill-cyl-atleast} implies that $u\in\Lambda(M,\Gamma,X)$ with $$\Area(u)\leq\Area(v)-\sum_{j=1}^k\mathcal{H}^2(Z_j)\leq a_p(\Gamma,X),$$
and thus $\Area(u)=a_p(\Gamma,X)$. Furthermore, the composition $P_\Gamma\circ v$ agrees with $v$ on the complement of $v^{-1}(Z)=u^{-1}(\Gamma)$, hence $u$ is infinitesimally isotropic on $M\setminus u^{-1}(\Gamma)$ with respect to~$g$. This concludes the proof of the theorem in the general case.
 \end{proof}
 \subsection{Translation to the smooth setting}
\label{sec:translation}
To obtain Theorem~\ref{thm:main0}, we make the following observations, where $M\in\mathcal{M}(k)$, $(X,h)$ is a complete Riemannian manifold and $u\in W^{1,2}(M,X)$.
\begin{itemize}
    \item By the Hopf-Rinow theorem, $X$ defines a proper geodesic metric space.
\item Homogeneously regular Riemannian manifolds admit a local quadratic isoperimetric inequality. See \cite{Jos85} for the definition and compare Section~4.3 in~\cite{Cre20} for the simple argument.
 \item Smooth Riemannian manifolds are locally of curvature bounded above, compare for example~\cite[Theorem~II.1A.6]{BrH99}.
\item  Compact $C^2$ curves in smooth Riemannian manifolds have finite total curvature, see~\cite{CFM10}. %{\color{red}See \cite{Sta:ar} for the definition of finite total curvature for curves in CAT(0) spaces.}
\item As a consequence of \eqref{eq:apmd-weakdiff}, for almost every $z\in M$ the approximate metric derivative $\apmd u_z$ defines a Euclidean seminorm on $T_zM$, and hence $u$ is infinitesimally isotropic if and only if it is weakly conformal.
\item Weakly conformal area minimizers in $X$ are minimizers of the Dirichlet energy, and thus weakly harmonic in the classical sense. Continuous weakly harmonic maps between Riemannian manifolds are however smooth by~\cite[Theorem 9.4.1]{Jos17}.
\end{itemize}
With these observations at hand, Theorem~\ref{thm:main} is easily seen to imply Theorem~\ref{thm:main0}.
\section{Minimizers under the conditions of cohesion and adhesion}
\label{sec:adh}
Let $X$ be a complete metric space, $M$ a smooth compact and connected surface and $\eta>0$. A mapping $u\colon M\to X$ is said to be \textit{$\eta$-cohesive} if $u$ is continuous and
$$\ell(u(c))\geq \eta$$
for every non-contractible closed curve $c$ in $M$.

\bd
A family $\mathcal{F}$ of maps from $M$ to $X$ is said to satisfy the condition of cohesion if there exists $\eta>0$ such that every map in $\mathcal{F}$ is $\eta$-cohesive.
\ed

Now let $c\subset M$ be an embedded arc such that the endpoints of $c$ lie on $\partial M$ and let $u\colon M\to X$ be continuous. If the endpoints of $c$ lie on a single component $\partial M_j$, then they divide $\partial M_j$ into two components~$c^-$ and $c^+$, where the notation is chosen such that $\ell(u(c^-))\leq \ell(u(c^+))$. Let $\bar{\rho}\colon(0,\infty)\to (0,\infty)$ be a function such that~$\bar{\rho}(\rho)\leq \rho$ for every $\rho \in (0,\infty)$. We say that $u\colon M\to X$ is \emph{$\bar{\rho}$-adhesive} if $u$ is continuous and for every arc $c$ with endpoints in $\partial M$ and of image-length $\ell(u(c))\leq \bar{\rho}(\rho)$, one has that the endpoints lie in the same connected component of $\partial M$ and
$$\ell(u(c^-))<\rho.$$
\bd
A family $\mathcal{F}$ of maps from $M$ to $X$ is said to satisfy the condition of adhesion if there exists a function $\bar{\rho}\colon(0,\infty)\to (0,\infty)$ as above such that every map in $\mathcal{F}$ is $\bar{\rho}$-adhesive.
\ed

Let $\Gamma$ be a configuration of $k\geq 1$ rectifiable closed curves in $X$ and $M\in\mathcal{M}(k)$. Set $$e(M,\Gamma,X):=\inf\{E_+^2(u,g):u\in\Lambda(M,\Gamma,X), \ g \text{ a Riemannian metric on } M\}.$$
An \textit{energy minimizing sequence in} $\Lambda(M,\Gamma,X)$ is a sequence of pairs $(u_n,g_n)$ of mappings $u_n\in\Lambda(M,\Gamma,X)$ and Riemannian metrics $g_n$ on $M$ such that $$ E_+^2(u_n,g_n)\to e(M,\Gamma,X)$$ as $n$ tends to infinity.

\bt\label{thm:existence-area-min-coh-adh}
Let $X$ be a proper metric space and $\Gamma\subset X$ a configuration of $k\geq 1$ rectifiable closed curves. Let $M\in \mathcal{M}(k)$ be connected. If there exist an energy minimizing sequence in $\Lambda(M,\Gamma,X)$ satisfying the conditions of cohesion and adhesion, then there exist $u\in\Lambda(M,\Gamma,X)$ and a Riemannian metric $g$ on $M$ such that
$$E_+^2(u,g)=e(M,\Gamma,X).$$
For any such $u$ and $g$ the map $u$ is infinitesimally isotropic with respect to~$g$.
\et
 If $X$ is a complete Riemannian manifold, then energy minimizers are precisely weakly conformal area minimizers. For more general spaces~$X$ however, the relation is more complicated and energy minimizers need not be area minimizers,~see for example~\cite{LW17-en-area,LW15-Plateau}. Nevertheless, one can obtain existence of area minimizers for singular configurations in proper metric spaces if there exists an area minimizing sequence satisfying the conditions of cohesion and adhesion by modifying the proofs of \cite[Theorem 1.6]{FW-epsConf} and \cite[Proposition 5.3]{FW-epsConf} accordingly. However, as in \cite[Theorem 1.6]{FW-epsConf} and \cite[Proposition 5.3]{FW-epsConf}, either the obtained area minimizers are potentially not infinitesimally isotropic, or one has to choose a somewhat different interpretation of the term 'area'.
 
\begin{proof}[Proof of Theorem~\ref{thm:existence-area-min-coh-adh}] 
It follows from~\cite[Corollary 1.3]{FW-epsConf} that any energy minimizing pair $(u,g)$ is infinitesimally isotropic. Thus it remains to show existence of such a pair.

First assume that $M$ is not of disc-type. If $\Gamma$ is a configuration of disjoint Jordan curves, then any continuous $u\in \Lambda(M,\Gamma,X)$ satisfies a $\rho'_\Gamma$-condition of adhesion, where $\rho'_\Gamma$ is as in Section~\ref{sec:reductionsfillings}. In fact, under this observation, the proof of Theorem~\ref{thm:existence-area-min-coh-adh} for such $M$ is a straightforward generalization of the proof of \cite[Theorem 8.2]{FW-Plateau-Douglas}. Namely, if one replaces in the statements of Propositions~8.3 and~8.4 in \cite{FW-Plateau-Douglas} the assumption that $\Gamma$ consists of disjoint Jordan curves by the assumption that $u$ is $\bar{\rho}$-adhesive, the proofs become virtually identical upon replacing $\rho'=\rho'_\Gamma$ by~$\bar{\rho}$. With these modified propositions at hand, the proof of Theorem~\ref{thm:existence-area-min-coh-adh} is completed as is that of \cite[Theorem 8.2]{FW-Plateau-Douglas}.

Finally assume that $\Gamma$ is a single curve and that $M=\bar{D}$. If $\Gamma$ is constant, the result is trivial. Otherwise we may represent $\Gamma$ as a composition of 3 curves $\Gamma_1,\Gamma_2, \Gamma_3$ of equal length. We also decompose $S^1$ into three consecutive arcs $\bar{\Gamma}_1$, $\bar{\Gamma}_2,\bar{\Gamma}_3$ of equal length. We say that a continuous map $u\in \Lambda(M,\Gamma,X)$ satisfies the \emph{$3$-arc condition} if $u|_{\bar{\Gamma}_i}$ is a parametrization of $\Gamma_i$ for every $i=1,2,3$. Fix $K\geq 0$ and adhesiveness function $\bar{\rho}\colon(0,\infty)\to (0,\infty)$. Let $\mathcal{F}$ be the family of maps $u\in\Lambda(M,\Gamma,X)$ which are $\bar{\rho}$-adhesive, satisfy the 3-arc condition and have energy $E_+^2(u,g_{\mathrm{Eucl}})\leq K$. We claim that the trace family $\{u|_{S^1}:u~\in\mathcal{F}\}$ is equicontinuous. To prove this claim, we fix $0<\varepsilon <\ell(\Gamma)/3$, $p\in S^1$ and $u\in \mathcal{F}$. Let $0<\delta <1$ be so small that 
\[
\pi \left(\frac{2K}{|\log \delta |}\right)^{\frac{1}{2}}<\bar{\rho}(\varepsilon).
\]
For $0<r<1$, denote by $c_r$ the arc $\{z\in \bar{D}: |z-p|=r\}$. By the Courant-Lebesgue lemma, \cite[Lemma 7.3]{LW15-Plateau},  there is $r\in (\delta,\sqrt{\delta})$ such that $\ell(u\circ c_r)\leq \bar{\rho}(\varepsilon)$. The $\bar{\rho}$-adhesiveness then implies that $\ell(u\circ c_r^-)\leq \varepsilon$, and hence it follows from the $3$-arc condition together with the choice of $\varepsilon$ that $c_r^-=B(p,r)\cap S^1$. Thus, for any $x\in B(p,\delta)\cap S^1$, one has $d(u(x),u(p))\leq \varepsilon$. Since the choice of $\delta$ was independent of $u$ and $p$, the claimed equicontinuity follows.

Now let $(u_n,g_n)$ be an energy minimizing sequence which is $\bar{\rho}$-adhesive. Since all metrics on the disc are conformally equivalent, we may assume that $g_n=g_{\mathrm{Eucl}}$ for each $n\in \mathbb{N}$. Furthermore, after precomposing with Moebius transforms, one has that all $u_n$ satisfy the $3$-arc condition. Thus by the claim the sequence $(u_n|_{S^1})$ is equicontinuous and hence Proposition~\ref{prop:Rell-Kond-comp} implies the existence of the desired energy  minimizer.
\end{proof}
\section{Appendix}
In this section we discuss the proof of Lemma~\ref{prop:thick}. A metric space $X$ will be called \emph{$\delta$-geodesic}, where $\delta>0$, if for all $x,y\in X$ satisfying $d(x,y)<\delta$ there is a curve $\gamma$ in $X$ joining $x$ to $y$ such that $\ell(\gamma)=d(x,y)$. Lemma~\ref{prop:thick} is only a slight strengthening of the following consequence of~\cite[Lemma 3.3]{LWYar}. 
\begin{lem}
\label{prop:thick2}
There is a universal constant~$C\geq 0$ such that for every proper, $\delta$-geodesic metric space $X$ and $0<\varepsilon\leq\delta$, there exists an $\varepsilon$-thickening~$Y$ of~$X$ such that $Y$ is proper and satisfies a $\left(C,\varepsilon/C \right)$-quadratic isoperimetric inequality.
\end{lem}
\cite[Lemma 3.3]{LWYar} is stated for spaces which are globally geodesic, though the proof readily gives the claimed result for $\delta$-geodesic spaces. Namely, in the proof the assumption only comes into play when estimating the diameter of the small ball $B_z$ with respect to its induced intrinsic metric by twice the radius. This estimate holds in a $\delta$-geodesic space as soon as the radius of the ball is bounded from above by~$\delta$. More precisely, this estimate is used twice: on p.~241 of~\cite{Wen08-sharp} to estimate the diameter of~$X_z$ and on p.~242 to find the curves $\bar{\gamma}_j$. \par
 For the proof of Lemma~\ref{prop:thick}, recall that the injective hull $E(X)$ of a compact metric space $X$ is a compact geodesic metric space. Furthermore, $X\subset E(X)$ isometrically and $\textnormal{diam}(E(X))=\textnormal{diam}(X)$,~see for example~\cite{Lan13}.
\begin{proof}[Proof of Lemma~\ref{prop:thick}]
We claim that for any $\delta>0$, there is an $(8\delta)$-thickening $Z$ of $X$ such that $Z$ is proper and $\delta$-geodesic. Lemma~\ref{prop:thick} then follows by first applying the claim to $X$, yielding a $(8C\varepsilon)$-thickening $Z$ of $X$ which is proper and $(C\varepsilon)$-geodesic, where $C$ is as in Lemma~\ref{prop:thick2}; and then applying Lemma~\ref{prop:thick2} to~$Z$ to obtain a $(C\varepsilon)$-thickening~$Y$ of~$Z$ which is proper and admits a $\left(C,\varepsilon \right)$-quadratic isoperimetric inequality. It remains to note that $Y$ is a $(9C \varepsilon)$-thickening of~$X$ and redefine $C$.\par 
In order to prove the claim, we perform a variation of the construction discussed in \cite{Wen08-sharp} and~\cite{LWYar}. Let $S$
be a maximal $\delta$-separated subset in $X$. For $z\in S$ set $B_z:=B(z,2\delta)$ and $X_z:=E(B_z)$. Then $\textnormal{diam}(B_z)\leq 4\delta$ and hence $\textnormal{diam}(X_z)\leq 4\delta$. We set
$$Z:= \Big(\bigsqcup_{z\in S} X_z\Big)_{\big/ \sim},$$
where $x\sim y$ if $x\in B_z\subset X_z$, $y\in B_w\subset X_w$ and $x=y$. The space $Z$ is endowed with the quotient metric. It follows from the construction that $Z$ is proper and a $(4\delta)$-thickening of $X$, compare also \cite{LWYar}.\par 
It remains to show that $Z$ is $\delta$-geodesic. To this end, let $x,y \in Z$ such that $d(x,y)<\delta$. Then either $x$ and $y$ lie in a common $X_z$ and $d(x,y)=d_{X_z}(x,y)$ or there are $z,w \in S$, $u\in X_z\cap X$ and $v\in X_w\cap X$ such that
$$d(x,y)=d_{X_z}(x,u)+d_X(u,v)+d_{X_w}(v,y).$$
In the former case, the distance is realized by a curve because~$X_z$ is geodesic. By the same reasoning, it suffices to show that $d(u,v)$ is realized by the length of a curve in $Z$ in the latter case. By maximality of $S$ there exists $s\in S$ such that $d_X(s,u)\leq \delta$ and hence $u,v\in X_s$. As $X_s\subset Z$ is a geodesic subset, the claim follows.
\end{proof}

%\bibliography{bibli-singPD}
%\printbibliography

\end{document}